\documentclass[a4paper,11pt,reqno]{amsart}
\usepackage{amsfonts}
\usepackage{amsmath}
\usepackage{tikz}
\usetikzlibrary{shapes.geometric, arrows.meta, positioning, fit, backgrounds}
\usepackage{logicproof}
\usepackage{amssymb, amsthm}
\usepackage{graphicx}
\usepackage{tabularx}
\usepackage{enumitem} 
\usepackage{caption}
\usepackage{capt-of}
\usepackage{booktabs}
\usepackage{float}
\usepackage{multirow,dsfont}
\usepackage{booktabs,subcaption}
\usepackage{siunitx}
\usepackage{amsthm}
\usepackage[dvipsnames]{xcolor}
\usepackage{array}
\usepackage{bm}
\usepackage{centernot}
\usepackage{enumitem}
\newcolumntype{L}[1]{>{\raggedright\let\newline\\\arraybackslash\hspace{0pt}}m{#1}}
\newcolumntype{C}[1]{>{\centering\let\newline\\\arraybackslash\hspace{0pt}}m{#1}}
\newcolumntype{R}[1]{>{\raggedleft\let\newline\\\arraybackslash\hspace{0pt}}m{#1}}
\usepackage[symbol]{footmisc}
\usepackage[labelfont=bf,format=plain,justification=raggedright,singlelinecheck=false]{caption}
\usepackage{mathrsfs}
\usepackage[]{mdframed}
\usepackage[colorlinks]{hyperref}
\usepackage{tcolorbox}
\usepackage{adjustbox}
\tcbuselibrary{theorems}
\newtcbtheorem[number within=section]{mytheo}{Note}%
{colback=white!5,colframe=black!35!black,fonttitle=\bfseries}{th}
\numberwithin{equation}{section}

\usepackage[margin=20mm]{geometry}
\usepackage{booktabs}
\usepackage{subcaption}
\usepackage{siunitx}
\usepackage{amsthm}
\theoremstyle{remark}
\newtheorem{remark}{Remark}

\usepackage{graphicx}
\usepackage{textgreek}

 {
      \theoremstyle{plain}
      \newtheorem{assumption}{Assumption}
  }
\newtheorem{theorem}{Theorem}[section]
\newtheorem{definition}[theorem]{Definition}

\usepackage{amsthm}

\theoremstyle{remark}
\makeatletter
\def\th@remark{%
  \normalfont 
  \thm@headfont{\bfseries} 
  \thm@notefont{\bfseries}%
  \thm@headpunct{.}%
}
\makeatother

\newtheorem{corollary}[theorem]{Corollary}
\newtheorem{proposition}{Proposition}

\newtheorem{lemma}[theorem]{Lemma}

\title{Scalable Operator Learning via Nyström Approximation With Denoising Applications}

\begin{document}

\maketitle

\author{}

\maketitle

\begin{center}

\author{ Naveen Gupta\textsuperscript{1}, Vaibhav Silmana\textsuperscript{2}, 
S. Sivananthan\textsuperscript{3}
}

\maketitle

\begin{center}

\textsuperscript{1} Department of Statistics, Pennsylvania State University, USA

\textsuperscript{2,3} Department of Mathematics, Indian Institute of Technology Delhi, India 

\textsuperscript{1}: \texttt{ngupta.maths@gmail.com}\\
\textsuperscript{2}: \texttt{vibhusilmana@gmail.com}\\
\textsuperscript{3}: \texttt{siva@maths.iitd.ac.in  }

\end{center}

\end{center}

\begin{abstract}
In this paper, we study Nyström subsampling for vector-valued regression in vector-valued reproducing kernel Hilbert spaces. Standard kernel methods often suffer from prohibitive computational costs due to the construction and inversion of large kernel matrices, which limits their scalability to large datasets. To overcome this bottleneck, we propose an efficient operator learning algorithm based on Nyström subsampling that accommodates functional outputs. Under general source conditions characterized by index functions—extending beyond the classical Hölder-type and operator-monotone frameworks—we establish minimax-optimal convergence rates for the proposed estimator. As an application of the proposed framework, we consider function denoising problems. Unlike classical denoising methods, which are typically tailored to specific signal representations or noise models, our approach formulates denoising within a general operator learning framework. Numerical experiments on signal denoising, real-time audio denoising, image denoising, inverse Radon transform reconstruction, and energy-efficiency prediction confirm that the proposed method achieves performance comparable to full kernel methods while substantially reducing computational cost.


\end{abstract}

\section{Introduction}

{
Denoising is among the most fundamental problems in machine learning, with applications including audio restoration, image reconstruction, medical imaging, motion sensors etc.
In the context of signal denoising, the objective is to recover a clean waveform from a corrupted/perturbed
signal. Similarly in image denoising, the goal is to reconstruct a clean image from perturbations arising from focus blur, motion blur, or Gaussian noise. A formal mathematical framework for denoising seeks to recover the original signal or image $X$ from corrupted observations $Y$. The degradation is commonly modeled as}
\[
Y = X + \epsilon \quad \text{or} \quad Y = H * X,
\]
where $\epsilon$ denotes additive noise and $H$ is a  blur kernel.

Over the years, numerous denoising algorithms have been developed by exploiting structural properties of signals and images, such as sparsity, low-rank structure, and hierarchical feature representations. These ideas have led to several influential classes of methods, including wavelet-based techniques, singular value decomposition (SVD)-based approaches, and deep convolution neural network (DnCNN) models, which we briefly review below.

Wavelet-based methods constitute one of the most influential and widely used frameworks for denoising, owing to their capacity to represent signals and images sparsely in the transform domain \cite{Donoho1994,Donoho1995}. This sparse representation enables a clear separation between significant signal features and noise, allowing effective noise suppression while preserving important structural information. Krim et al. \cite{wavelet_denosiing_2000} proposed a risk-based best-basis denoising method for signals corrupted by additive white Gaussian noise. The method searches families of orthonormal bases, such as wavelet packet and local cosine bases, and selects the basis that minimizes an estimated mean-square reconstruction error after hard thresholding of noisy coefficients. The selected basis is then used to threshold the coefficients and reconstruct the signal. Thus, the framework combines adaptive basis selection with denoising, while associating the reconstruction with an explicit risk/mean-square-error estimate. Sahoo et al.~\cite{optimal_wavelet} introduced a sparsity-based empirical approach for optimal mother-wavelet selection in signal denoising. The method analyzes the sparsity of wavelet detail coefficients across decomposition levels, determines an optimal decomposition level separating noisy and noise-free detail components, and uses the mean of sparsity change, \textcolor{red}{$\mu_{sc}$}, as a quantitative criterion to select the wavelet or group of wavelets with the highest \textcolor{red}{$\mu_{sc}$} values. These selected wavelets are interpreted as maximizing the separation between signal-dominated and noise-dominated components in the wavelet domain, thereby enabling automated wavelet selection for effective denoising.

Al-Tawaha et al.~\cite{Svd_comparison} proposed a non-iterative SVD-based signal denoising method using two independent noisy observations of the same signal. The method forms one Hankel matrix from the average of the two observations and another from half their difference. Using a random-matrix universality argument together with Weyl’s inequality, they show that, asymptotically, the largest singular value of the difference Hankel matrix provides an upper bound on the spurious noise-induced singular values of the average Hankel matrix. Singular values of the average Hankel matrix below this threshold are set to zero, and the denoised signal is reconstructed from the resulting reduced-rank Hankel matrix.

{
Dabov et al.~\cite{Bm3d} presented BM3D, a highly influential image-denoising algorithm based on sparse 3D transform-domain collaborative filtering. The method groups similar 2D image blocks into three-dimensional arrays and performs collaborative filtering using 3D transform-domain shrinkage, implemented through a two-step procedure with collaborative hard-thresholding followed by collaborative Wiener filtering. By exploiting both intra-block and inter-block correlations, BM3D achieves a highly sparse representation of the underlying image and delivers state-of-the-art denoising performance. Zhang et al. \cite{DnCNN-paper} proposed a DnCNN for image denoising that employs residual learning to estimate the noise component rather than directly reconstructing the clean image. By integrating residual learning with batch normalization, the network achieves efficient training and improved denoising performance. Furthermore, DnCNN supports blind Gaussian denoising by learning a single model over a range of noise levels, eliminating the need for prior noise-level estimation.
Liu et al.~\cite{blurring} proposed a spectral approach to blind image deblurring based on the spectral properties of images viewed as convolution operators. Motivated by the observation that sharp images are often high-pass whereas blurry images are low-pass, they derive a convex blur-kernel regularizer that depends only on the observed blurry image, without requiring knowledge of the latent sharp image. This regularizer is incorporated into an alternating minimization framework, where the blur kernel is updated using the spectral regularizer and the latent image is updated via total-variation-based non-blind deconvolution. The method can estimate unknown blur kernels and recover sharp images in challenging, severely blurred cases, though its effectiveness depends on the presence of sufficiently sharp image structures and is limited under substantial noise.}

{
Underlying both audio and image denoising is a shared mathematical structure: the
desired output is not a scalar prediction but an entire function or high-dimensional
structured object. These tasks are canonical instances of a broader and increasingly
important class of problems that require the prediction of \emph{structured outputs},
naturally formulated as \emph{vector-valued regression}, in which one seeks to learn
a mapping from an input space~$\mathcal{X}$ to a high-dimensional---or potentially
infinite-dimensional---output space~$\mathcal{Y}$.}

{ 
The classical scalar-valued regression framework is fundamentally
insufficient for these settings. In contrast, vector-valued regression directly models
the joint output structure, enabling richer predictions over function
spaces, matrix-valued responses, and operator-valued
outputs~\cite{VRKHS(BASIC), holzleitner2026regularizedlearningfunctionaldata,
neural_network_vvr, VRKHS_REF, meunier2024optimal, nnetworks_vvr}, with
the vector-valued reproducing kernel Hilbert space (vRKHS) providing a
principled theoretical foundation for such settings.}


In this work, we put our focus on kernel based methods for vector-valued regression framework. 
Through the \emph{kernel trick}, inputs are implicitly mapped into a 
potentially infinite-dimensional feature space, and the \emph{representer 
theorem}~\cite{Representor_kimeldrof_wahba} guarantees a finite representation of 
the empirical risk minimizer in terms of the training data. In a seminal contribution, Meunier et al. ~\cite{meunier2024optimal} introduced 
a practical algorithm for operator learning via \emph{general 
multiplicative kernels}. By extending the representer theorem of~\cite{Representor_kimeldrof_wahba} to general spectral filter functions, a unified treatment of operator learning algorithms is provided, admitting a closed-form empirical solution that remains tractable even when the output space $\mathcal{Y}$ is infinite-dimensional. However, kernel methods incur $\mathcal{O}(n^2)$ memory and 
$\mathcal{O}(n^3)$ time complexity, rendering them infeasible for large dataset. 
Several strategies have been proposed to address this bottleneck, including 
\emph{distributed learning}~\cite{Distributed_learning, Distributed_1}, 
which partitions data across machines while preserving statistical 
optimality, and \emph{greedy methods}~\cite{greedy}, which select a sparse 
set of basis functions. In this work, we adopt \emph{Nystr\"{o}m 
subsampling} for operator learning, which constructs a low-rank kernel approximation from 
$m \ll n$ subsample points, reducing time complexity to $\mathcal{O}(nm^2 + m^3)$ 
and storage to $\mathcal{O}(nm)$~\cite{nystrom_flr, shuai_lu, 
Myleiko2019RegularizedNS, Myleiko_covariate, less_is_more}.

To the best of our knowledge, the theoretical analysis of Nystr\"om subsampling has been largely confined to the scalar-valued or finite-dimensional output setting over the past decades.
 Minimax-optimal convergence rates were first 
established by Rudi et al. ~\cite{less_is_more} under 
classical H\"{o}lder-type source conditions. This was subsequently extended 
in two directions: Myleiko et al. ~\cite{Myleiko2019RegularizedNS} generalized to the splitting class 
of index functions with adaptive regularization, while Lu et al. ~\cite{shuai_lu} treated the operator monotone class in the 
\emph{misspecified} setting, where the regression function need not belong 
to the RKHS. More recently, the framework has been extended to 
\emph{covariate shift} \cite{covnystrom,Myleiko_covariate}. Myleiko and Solodky~\cite{Myleiko_covariate} 
studied uniformly bounded importance weights under general splitting-type 
regularization, while   Vecchia et al. \cite{covnystrom} established optimal rates for 
both bounded and unbounded weights under H\"{o}lder conditions in the 
misspecified setting.
{Despite significant progress in this area, the existing theoretical analyses 
of Nystr\"{o}m subsampling have been largely confined to scalar-valued 
output settings, or at most to 
finite-dimensional outputs.  While spectral regularization methods such as 
kernel ridge regression  have been analyzed in the 
vRKHS setting, the theoretical analysis of 
Nystr\"{o}m subsampling in the vRKHS setting remains largely 
underdeveloped, representing a significant gap between theory and the 
growing practical demand for scalable operator learning methods}.


In the vector-valued regression setting, Caponnetto and De Vito~\cite{devito(2007)} established optimal convergence rates for Tikhonov regularization under Hölder source conditions for kernels whose associated integral operators are trace-class. This line of work was subsequently extended by Rastogi and Sampath~\cite{rastogi2017optimal}, who derived optimal convergence rates for both Tikhonov and general regularization schemes under more general source conditions, while still operating within the trace-class framework. Despite their theoretical significance, these results do not directly lead to practical operator learning algorithms. For example, scalar multiplicative kernels admit a representer theorem with a closed-form solution, thereby yielding computationally tractable operator learning algorithms \cite{meunier2024optimal,first_vector_valued_algo}. However, the integral operators associated with such kernels are generally not trace-class and therefore
fall outside the theoretical framework of the aforementioned works.

To overcome this limitation, Meunier et al.~\cite{meunier2024optimal} introduced a novel framework for learning with general multiplicative kernels. Their approach reformulates the problem by shifting the search space from the hypothesis space  
to a suitable space of Hilbert--Schmidt operators, and establishes optimal convergence rates for general regularization schemes under Hölder source conditions. 
Motivated by the above work, we aim to develop a computationally efficient operator learning 
algorithm and establish optimal convergence rates under general source conditions, characterized 
by an index function.
Although the proposed framework is applicable to a broad class of operator learning problems, we focus on signal denoising as a representative application to demonstrate its practical effectiveness.

\vspace{0.3em}

\textbf{Main Contributions.} In this work, we investigate the Nystr\"om subsampling approach for vector-valued regression with functional responses within the vRKHS framework. Specifically, the main contributions of this work are summarized as follows:

\begin{enumerate}[label=(\roman*)]
    \item We design a Nystr\"{o}m-based operator learning algorithm in the vRKHS setting, reducing the time complexity of kernel methods to $\mathcal{O}(nm^2)$ and memory complexity to $\mathcal{O}(nm)$ while accommodating infinite-dimensional outputs.

    \item We establish minimax-optimal convergence rates under a general class of source 
    conditions characterized by arbitrary index functions, subsuming  H\"{o}lder, 
    operator monotone and splitting form of index function assumptions as special cases.

    \item  We validate our method on signal denoising, real-time audio denoising, image denoising, inverse randon transform reconstruction, and  energy 
    efficiency prediction --- demonstrating 
    accuracy comparable to full kernel methods at significantly reduced computational cost.
   
\end{enumerate}

The remainder of this paper is organized as follows. In Section~\ref{section:2}, 
we introduce the necessary preliminaries and notation required for our analysis, 
and present the construction of our proposed Nystr\"{o}m algorithm for vector-valued 
operator learning. Section~\ref{section:3} provides theoretical convergence rates 
under general source condition. 
Section~\ref{section:4} reports numerical experiments on a variety of tasks, illustrating 
the practical effectiveness of our method.

\section{ Kernel Methods and Nyström Approximations: Background and Proposed Algorithm}\label{section:2}

In this section, we provide the mathematical structure of our model and present the necessary preliminaries required for the analysis, followed by our proposed algorithm.

\noindent
Let $\mathcal{X}$ be a second countable locally compact Hausdorff space equipped with a Borel $\sigma$-field $\mathcal{F}_\mathcal X$ and $\mathcal Y$ be a real separable Hilbert space $(\mathcal {Y}, \langle \cdot, \cdot \rangle_{{\mathcal Y}})$. It is assumed that the input space $\mathcal{X}$ and output space $\mathcal{Y}$ are related by some joint distribution $\rho$ on $\mathcal{X}\times \mathcal{Y}$ and under some conditions this probability distribution $\rho$ splits into marginal distribution $\rho_\mathcal X$ and conditional distribution $\rho(y|x)$ i.e., 
$$\rho(x,y) =\rho_\mathcal X(x)\rho(y|x),~\forall~ x \in \mathcal  X, y \in \mathcal Y.$$

\noindent
{The goal is to find a function $f:\mathcal{X} \to \mathcal Y$ which predicts the response $y$ for an unseen input $x$. To measure the error for such a function} $f :\mathcal{X} \to \mathcal Y$, we define the \emph{expected risk} of choosing $f$ as a predictor by
$$\mathcal{R}(f) =  \int_{\mathcal{X}\times \mathcal{Y}} \|f(x) - y\|_\mathcal Y^2 \, d\rho(x, y).$$

\noindent
Under the assumption 
$\int_{\mathcal X \times \mathcal Y} \|y\|_\mathcal{Y}^{2} \, d\rho(x,y) < \infty$, we can see that the regression function defined as 
\[
    f_{\rho}(x) \;=\; \int_{\mathcal{Y}} y \, d\rho(y \mid x),
    \qquad x \in \mathcal{X},
\]
lies in $L^{2}(\rho_{X}; \mathcal{Y})$
and minimizes the \emph{expected risk} $\mathcal{R}(f)$ over $L^{2}(\rho_{X}; \mathcal{Y})$. Note that, \(L^{2}(\rho_{X};\mathcal Y)\) denotes the Bochner space of square-integrable functions with respect to the measure \(\rho_\mathcal X\), taking values in a Hilbert space \( \mathcal Y\).\\

\noindent
In practical situations, only information about $\rho$ is known through a finite set of $n$ independent observations denoted as
\begin{center}
    $z = \{ (x_1, y_1), \dots, (x_n, y_n) \} $,
\end{center}
and hence neither the expected risk $\mathcal{R}(f)$ nor the regression function $f_\rho$ can be computed directly.
To address this issue, we consider the Tikhonov regularization based empirical risk minimization problem:

\begin{equation}\label{risk}
    \min_{f \in \Theta}\mathcal{R}_z(f)~,~\qquad \mathcal{R}_z(f) = \frac{1}{n}\sum_{i=1}^n \| f(x_i) - y_i \|_{\mathcal{Y}}^2 + \lambda \|f\|^2_\mathcal{G} ,
\end{equation}
over a pre-defined hypothesis space $\Theta$.  One important class of hypothesis spaces  which we will use in this work is vRKHS. Next we will be giving a short introduction to vRKHS.






\vspace{0.8em}
\subsection{Construction of vRKHS}

\noindent
Let \( \mathcal{X} \) be a non-empty set and let \( \mathcal{Y} \) be a real separable Hilbert space equipped with inner product \( \langle \cdot,\cdot \rangle_{\mathcal Y} \).  
Consider an operator-valued kernel
\[
K : \mathcal{X} \times \mathcal{X} \to \mathcal L(\mathcal{Y}),
\]
where \( \mathcal L(\mathcal Y) \) denotes the space of bounded linear operators on \( \mathcal Y \).  
We assume that \( K \) is symmetric and positive-definite, meaning that for every \( n \in \mathbb N \), any collection of points \( \{x_i\}_{i=1}^n \subset \mathcal X \), and vectors \( \{y_i\}_{i=1}^n \subset \mathcal Y \),
\[
\sum_{i=1}^n \sum_{j=1}^n 
\left\langle y_i, K(x_i,x_j)y_j \right\rangle_{\mathcal Y} \geq 0.
\]

\noindent
Associated with the kernel \(K\), we first define the pre-Hilbert space
\[
\mathcal G_{\mathrm{pre}}
:=
\mathrm{span}\left\{
K_x y := K(\cdot,x)y
\;\middle|\;
x\in\mathcal X,\; y\in\mathcal Y
\right\},
\]
consisting of finite linear combinations of kernel sections.  
The inner product on \( \mathcal G_{\mathrm{pre}} \) is introduced through the kernel itself by setting
\[
\left\langle K_x y,\; K_{x'} y' \right\rangle_{\mathcal G}
:=
\left\langle y,\; K(x,x')y' \right\rangle_{\mathcal Y},
\]
and extending bilinearly to all elements of \( \mathcal G_{\mathrm{pre}} \).

\noindent
The vector-valued reproducing kernel Hilbert space associated with \(K\), denoted by \( \mathcal G \), is then defined as the completion of \( \mathcal G_{\mathrm{pre}} \) with respect to this inner product.  
Moreover, there exists a one-to-one correspondence between operator-valued positive-definite kernels and vRKHSs~\cite{first_vector_valued_algo}. For a detailed and comprehensive study of vRKHS, we refer readers to see~\cite{VRKHS(BASIC), carmeli2008vectorvaluedreproducingkernel,pedrick1957reproducing,schwartz_VRKHS_ORIGIN} and references therein.

\vspace{0.4em}

\noindent
{To dive deep into the theoretical analysis of kernel method, we need to define these following operators.} Let \(X= \{ x_1, x_2, \dots, x_n \} \) and \(\textbf{y}= \{ y_1, y_2, \dots, y_n \} \) be an ordered set.  
We define the \emph{sampling operator}
\[
S_{X} : \mathcal{G} \longrightarrow \mathcal Y^n, 
\qquad 
S_{X}(f) = \bigl(f(x_1), f(x_2), \dots, f(x_n)\bigr),
\]
and its \emph{adjoint operator}
\[
S_{X}^* :\mathcal Y^n \longrightarrow \mathcal{G}, 
\qquad 
S_{X}^*(\mathbf{y}) = \frac{1}{n} \sum_{i=1}^n  K_{x_i} y_i,
\]
where $\mathcal{Y}^{n} = \underbrace{\mathcal{Y} \times \mathcal{Y} \times \cdots\times \mathcal{Y}}_{n-\text{times}}$.

\noindent
Taking Fréchet derivative and putting it against, it can be easily seen that the empirical minimizer \( \widehat f_{z,\lambda} \) of the empirical risk~\eqref{risk} over the hypothesis space \( \mathcal{G} \) satisfies
{\begin{equation}\label{equation:general f_Z}
 \widehat f_{z,\lambda} \;=\;\bigl(S_{X}^* S_{X} +\lambda I \bigr)^{-1} S_{X}^* \mathbf{y}.
\end{equation}}
\vspace{0.4mm}

\noindent
Throughout this paper, we will consider that the space vRKHS $\mathcal{G}$ is induced by a kernel of the form
\begin{equation}\label{multk}
K(x,t) \;=\; k(x,t)\, Id_\mathcal Y,    
\end{equation}
(such kernel is referred as scalar multiplicative kernel) where $Id_\mathcal Y :\mathcal Y \to\mathcal Y$ denotes the identity operator on $\mathcal Y$ and $k : \mathcal{X} \times \mathcal{X} \to \mathbb{R}$ is a scalar valued reproducing kernel. The scalar valued RKHS induced by kernel $k$ is denoted by $\mathcal{H}$.\\

\noindent
{As the associated integral operator for the multiplicative kernel is not compact even under the standard assumption $\sup_{x\in \mathcal{X}}k(x,x)\leq \kappa^2$, the analysis can not be done in usual framework with classical techniques. Next we state an important theorem that establishes an isomorphism between the vRKHS $\mathcal{G}$ and the Hilbert space of Hilbert--Schmidt operators from $\mathcal{H}$ to $\mathcal Y$, denoted by $\mathcal{S}_2(\mathcal{H},\mathcal Y)$. We use the notation $\mathcal{S}_2(\mathcal{H})$ to denote the space of Hilbert--Schmidt operators from $\mathcal{H}$ to $\mathcal{H}$.}

\begin{theorem}[\cite{VRKHS_REF}, Corollary 1]\label{vrkhs-isomorphism} Let $\phi:\mathcal X\rightarrow \mathcal{H}$ denotes the feature map such that $\phi(x) = k(x,\cdot)$. Then for every function $f \in \mathcal{G}$ there exists a unique operator 
$C \in \mathcal S_2(\mathcal{H}, \mathcal{Y})$ such that 

\begin{center}
    $f(\cdot) = C \phi(\cdot) $
\end{center}

\noindent
with 

\begin{center}
    $\|C\|_{\mathcal S_2(\mathcal{H}, \mathcal{Y})} = \|f\|_{\mathcal{G}}$
\end{center}

\noindent
and vice versa. Hence $\mathcal{G} \simeq \mathcal S_2(\mathcal{H}, \mathcal{Y})$ and
it follows that $\mathcal{G}$ can be written as

\begin{center}
    $\mathcal{G} = \{ f : \mathcal{X} \to \mathcal{Y} \mid 
    f = C \phi(\cdot), \; C \in \mathcal S_2(\mathcal{H}, \mathcal{Y}) \}.$
\end{center}
\end{theorem}

Next we state the representer theorem for the minimizer of \eqref{risk} established in \cite{meunier2024optimal} for scalar multiplicative kernels.

\begin{proposition}\label{rep-general}
Let $(\mathbf{K})_{ij} = k(x_i, x_j)$, $1 \leq i,j \leq n$, denote the Gram matrix associated with the scalar-valued kernel $k$. 
Then, the estimator $\widehat{f}_{z,\lambda}$ admits the representation
\begin{equation}
    \widehat{f}_{z,\lambda}(x)
    \;=\;
    \sum_{i=1}^{n} y_i\, \alpha_i(x),
    \qquad
    \alpha(x)
    \;=\;
    \, \left( {\mathbf{K}} +\lambda n I \right)^{-1} \mathbf{k}_x,
    \qquad
    (\mathbf{k}_x)_i = k(x, x_i), \quad 1 \leq i \leq n.
    \label{eq:representer1}
\end{equation}
\end{proposition}

\noindent
Since the Gram matrix $\mathbf{K}$ is symmetric, it follows that 
$\left( {\mathbf{K}} +\lambda n I \right)^{-1} $ is symmetric as well. 
Using this property, the estimator $\widehat{f}_{z,\lambda}$ can equivalently be expressed as
\begin{equation}
    \widehat{f}_{z,\lambda}(x)
    \;=\;
    \sum_{i=1}^{n} k_{x_i}(x)\, \alpha_i,
    \qquad
    \alpha
    \;=\;
    \left( {\mathbf{K}} +\lambda n I \right)^{-1}  \mathbf{y},
    \qquad
    \mathbf{y} = (y_1, \ldots, y_n)^\top.
    \label{eq:representer2}
\end{equation}

\noindent
A major limitation of the kernel methods lies in its computational burden. 
Here, the construction of the estimator requires to invert an \( n \times n \) size Gram matrix which entails a time complexity of 
\(\mathcal{O}(n^3)\) and a memory complexity of \(\mathcal{O}(n^2)\). 
Such requirements can become prohibitive as the dataset size grows, thereby 
limiting the scalability of the method. To overcome this limitation of kernel methods, we study the Nyström subsampling method in the vRKHS framnework.

\subsection{Nyström subsampling method}
Let $\widetilde{X} = \{\tilde{x}_1, \tilde{x}_2, \ldots, \tilde{x}_m\}$ denote a set of $m$ points sampled uniformly from the training inputs $\{x_1, x_2, \ldots, x_n\}$, and let $\widetilde{Y} = \operatorname{span}\{y_1, \ldots, y_n\} \subseteq \mathcal{Y}$ be the subspace spanned by the training outputs. We define the Nystr\"{o}m subspace $\widetilde{\mathcal{G}}_m \subseteq \mathcal{G}$ by
\[
\widetilde{\mathcal{G}}_m = \operatorname{span}\bigl\{ k(x, \cdot)\, h : x \in \widetilde{X},\; h \in \widetilde{Y} \bigr\},
\]
that is, the linear span of all elements of the form $k(x, \cdot)h$, where $x$ ranges over $\widetilde{X}$ and $h$ ranges over $\widetilde{Y}$. The main idea of the Nystr\"{o}m method is to work with this smaller subspace $\widetilde{\mathcal{G}}_{m}$ instead of the full vRKHS $\mathcal{G}$, and to formulate the following empirical risk minimization problem:

\begin{equation}\label{p1}
\widehat{f}_{m,\lambda} := \arg\min_{f \in \widetilde{\mathcal{G}}_m} \frac{1}{n} \sum_{i=1}^n \| f(x_i) - y_i \|_{\mathcal{Y}}^2 + \lambda \|f\|_{\mathcal{G}}^2.
\end{equation}

\noindent
By applying Theorem~\ref{vrkhs-isomorphism} to 
the empirical risk functional in~\eqref{p1}, we obtain
\begin{equation}\label{equation:emphiirical_Riskin C}
\mathcal{R}_{z}(C) 
    \;:=\; \frac{1}{n}\sum_{i=1}^n 
    \left\| y_i - C \phi(x_i) \right\|_{\mathcal{Y}}^2 
    + \lambda \|C\|_{\mathcal{S}_2(\mathcal{H},\mathcal{Y})}^2.
\end{equation}
Moreover, the minimization problem over $\mathcal{G}$ can now be equivalently 
expressed as a minimization over the space of Hilbert--Schmidt operators:
\[
\min_{f \in \mathcal{G}} \; \mathcal{R}_{z}(f)
\;=\;
\min_{C \in \mathcal{S}_2(\mathcal{H},\mathcal{Y})} \; \mathcal{R}_{z}(C).
\]

By recasting the problem in the space of Hilbert--Schmidt operators, 
we effectively transform the search for an optimal function 
$f \in \mathcal{G}$ into the task of identifying the best operator 
$C \in \mathcal S_2(\mathcal{H}, \mathcal{Y})$ that minimizes the same risk functional. Let $C_* \in \mathcal S_2(\mathcal{H},\mathcal Y)$ denotes the minimizer of $\mathcal{R}_{\mathbf z}(C)$ over $\mathcal S_2(\mathcal{H},\mathcal Y)$, then the minimizer over $\mathcal{G}$  is given by
\[
f^\ast(\cdot)=C_* \phi(\cdot).
\]
  
\noindent
We briefly recall the concept of the tensor product of Hilbert spaces 
before presenting several lemmas that are essential for the construction 
of our algorithm. The reader is referred to \cite{aubin} for a thorough treatment.

\medskip

Let $\mathcal{H}$ and $\mathcal{H}'$ be two Hilbert spaces. The tensor product space $\mathcal{H} \otimes \mathcal{H}'$ is the Hilbert space generated by finite linear combinations of elementary tensors of the form $x \otimes x'$, where $x \in \mathcal H$ and $x' \in \mathcal H'$.
Each elementary tensor $x \otimes x'$ is identified with the rank-one operator from $\mathcal H'$ to $\mathcal H$ given by
\begin{equation*}
    x \otimes x': \mathcal{H}' \to \mathcal{H}, ~\qquad~~~~~ y'\mapsto \langle y',x' \rangle_{\mathcal H'}x,
\qquad y' \in \mathcal H'.
\end{equation*}

Furthermore, $\mathcal H \otimes \mathcal H'$ is  isometrically isomorphic to the Hilbert space of Hilbert--Schmidt operators $\mathcal S_2(\mathcal H',\mathcal H)$. Hence, we identify these spaces and use the notations interchangeably throughout this work.

\noindent

\noindent
\begin{lemma}\label{Gm=Gm'lemma}
Let
$\mathcal{H}_m
=
\operatorname{span}\{\phi(\tilde{x}) : \tilde{x}\in \widetilde{X}\}
\subseteq \mathcal H$,
and $P:\mathcal H \to \mathcal H_m$
be the orthogonal projection onto \(\mathcal H_m\). Define the class
\[
\mathcal{G}_m'
=
\left\{
f:\mathcal X \to \mathcal Y
\;\middle|\;
f(x)= C_m(P\phi(x))
\text{ for some }
C_m \in \mathcal S_2(\mathcal H_m,\widetilde{\mathcal Y})
\right\}.
\]

\noindent
Then the function class \(\mathcal G_m'\) coincides with
\(\widetilde{\mathcal G}_m\), i.e., 
\[
\mathcal G_m'=\widetilde{\mathcal G}_m.
\]
\end{lemma}

\begin{proof}
Let $f \in \widetilde{\mathcal{G}}_m$. Then there exist $h_1,\dots,h_m \in \tilde Y$ such that
\begin{align*}
f(x)
&= \sum_{i=1}^m h_i k(x,\tilde{x}_i) = \sum_{i=1}^m h_i 
   \langle \phi(x), \phi(\tilde{x}_i) \rangle_{\mathcal{H}}.
\end{align*}
Note that $\phi(\tilde{x}_i) \in \mathcal{H}_m$, hence we have
\[
\langle \phi(x), \phi(\tilde{x}_i) \rangle_{\mathcal{H}}
=
\langle P\phi(x), \phi(\tilde{x}_i) \rangle_{\mathcal{H}},
\]
then it follows that
\[
f(x)
=
\sum_{i=1}^m h_i
\langle P\phi(x), \phi(\tilde{x}_i) \rangle_{\mathcal{H}}.
\]

\noindent
To show that $\widetilde{\mathcal{G}}_m \subseteq \mathcal G_m'$, consider a linear operator $C_m : \mathcal{H}_m \to \tilde Y$ defined as
\[
C_m(g)
=
\sum_{i=1}^m
h_i \,
\langle g, \phi(\tilde{x}_i) \rangle_{\mathcal{H}},
\qquad g \in \mathcal{H}_m. 
\]
Then
\[
C_m P\phi(x)
=
\sum_{i=1}^m
h_i
\langle P\phi(x), \phi(\tilde{x}_i) \rangle_{\mathcal{H}}
=
f(x),
\]
and therefore $f \in \mathcal{G}_m'$.\\

\noindent
Conversely, let $f \in \mathcal{G}_m'$. Then there exists an operator 
$C_m \in \mathcal{S}_2(\mathcal{H}_m, \tilde Y)$ such that
\[
f(x) = C_m P\phi(x).
\]
Since $
\mathcal{H}_m = \operatorname{span}\{\phi(\tilde{x}_1),\dots,\phi(\tilde{x}_m)\}
$ is finite-dimensional, every Hilbert--Schmidt operator on $C_m\in \mathcal{S}_2(\mathcal{H}_m,\tilde Y)$ admits a finite expansion. Hence there exist vectors $h_1,\dots,h_m \in \tilde Y$ such that
\[
C_m(g)
=
\sum_{i=1}^m
h_i \,
\langle g, \phi(\tilde{x}_i) \rangle_{\mathcal H},
\qquad \forall\, g \in \mathcal{H}_m,
\]
where each coefficient vector satisfies
\[
h_i = \sum_{j=1}^n a_{i,j}\, y_j.
\]
Applying this to $g = P\phi(x)$ yields
\begin{align*}
f(x) &= \sum_{i=1}^m h_i
\langle P\phi(x), \phi(\tilde{x}_i) \rangle_{\mathcal{H}} =
\sum_{i=1}^m h_i
\langle \phi(x), \phi(\tilde{x}_i) \rangle_{\mathcal{H}} =
\sum_{i=1}^m
h_i k(x,\tilde{x}_i) \in \tilde{\mathcal{G}}_{m}.
\end{align*}
\end{proof}

\noindent
Using Lemma \ref{Gm=Gm'lemma} along with Theorem \ref{vrkhs-isomorphism} and \eqref{p1}, we have

\begin{equation}\label{error_C}
\min_{f \in \tilde {\mathcal{G}}_m} \frac{1}{n} \sum_{i=1}^n \| f(x_i) - y_i \|_\mathcal Y^2 +\lambda \|f\|_\mathcal{G}^2 = \min_{C \in \mathcal{S}_2(\mathcal{H}_m,\tilde Y)} \frac{1}{n} \sum_{i=1}^n \| CP\phi(x_i) - y_i \|_\mathcal Y^2 +\lambda \|C\|_{\mathcal{S}_2(\mathcal{H},\mathcal{Y})}^2.
\end{equation}

\noindent
As we have converted our minimization problem over $\mathcal{S}_2(\mathcal{H}_m,\tilde Y)$ from $\mathcal{G}_{m}$, we provide with the estimator for \eqref{error_C} in the following lemma.

\begin{lemma}\label{reprenstorforproof}
The minimizer of the regularized empirical risk functional~\eqref{equation:emphiirical_Riskin C} is given by
\[
\widehat C_\lambda^P
=
\widehat{C}_{\mathcal Y,\mathcal X}^P
\bigl(\widehat{C}_{\mathcal X}^P+\lambda I\bigr)^{-1},
\]
where the empirical cross-covariance operator 
$\widehat{C}_{\mathcal Y,\mathcal X}^P$
and the empirical covariance operator 
$\widehat{C}_{\mathcal X}^P$
are defined by
\[
\widehat{C}_{\mathcal Y,\mathcal X}^P
=
\frac{1}{n}
\sum_{i=1}^n
y_i \otimes P\phi(x_i),
\qquad
\widehat{C}_{\mathcal X}^P
=
\frac{1}{n}
\sum_{i=1}^n
P\phi(x_i)\otimes P\phi(x_i).
\]
\end{lemma}

\begin{proof}
Let
\[
T(C)
=
\frac{1}{n}
\sum_{i=1}^n
\| C P\phi(x_i) - y_i \|_\mathcal Y^2 +\lambda \|C\|_{\mathcal{S}_2(\mathcal{H},\mathcal{Y})}^2,
\qquad 
C \in \mathcal{S}_2(\mathcal{H}_m,\tilde Y) .
\]

\noindent
For a perturbation $H \in \mathcal{S}_2(\mathcal{H}_m,\tilde Y)$, consider
\begin{align*}
T(C+H) - T(C)
&=
\frac{1}{n}
\sum_{i=1}^n
\Big(
\| (C+H)P\phi(x_i) - y_i \|_\mathcal Y^2
-
\| CP\phi(x_i) - y_i \|_\mathcal Y^2
\Big) +\lambda \|C+H\|_{\mathcal{S}_2(\mathcal{H},\mathcal{Y})}^2-\lambda\|C\|_{\mathcal{S}_2(\mathcal{H},\mathcal{Y})}^2 .
\end{align*}

\noindent
Expanding the square yields
\begin{align*}
T(C+H) - T(C)
&=
\frac{2}{n}
\sum_{i=1}^n
\langle CP\phi(x_i) - y_i,\,
H P\phi(x_i) \rangle_\mathcal Y +2\lambda\langle C,H\rangle_{\mathcal{S}_2(\mathcal{H},\mathcal{Y})}
+ \mathcal{O}(\|H\|^2).
\end{align*}

\noindent
Using the identity
\[
\langle a,\, Hb \rangle_\mathcal Y
=
\langle H,\, a \otimes b \rangle_{\mathcal{S}_2(\mathcal{H}, \mathcal Y)},
\]
we obtain
\begin{align*}
T(C+H) - T(C)
&=
\frac{2}{n}
\sum_{i=1}^n
\Big\langle
H,\,
(CP\phi(x_i) - y_i)\otimes P\phi(x_i)
\Big\rangle_{\mathcal{S}_2(\mathcal{H},\mathcal Y)} +2\lambda\langle H,C\rangle_{\mathcal{S}_2(\mathcal{H},\mathcal{Y})}
+ \mathcal{O}(\|H\|^2).
\end{align*}

\noindent
Hence, taking the Fréchet derivative of $\mathcal{R}_\mathbf{z}(C)$ and setting it equal to zero, we obtain
\[
\frac{1}{n}
\sum_{i=1}^n
CP\phi(x_i) \otimes P\phi(x_i) + \lambda C
=
\frac{1}{n}
\sum_{i=1}^n
y_i \otimes P\phi(x_i).
\]

\noindent
Thus, we obtain
\[
C (\widehat{C}_\mathcal{X}^P +\lambda I )
= \widehat{C}_{\mathcal Y, \mathcal X}^P.
\]

\noindent
Hence, the minimizer of the empirical risk satisfies the desired equation.

\end{proof}

\noindent
We now present two forms of our operator learning algorithm for Nyström subsampling, both of which reduce time and memory complexity while matching the accuracy of full KRR.

\vspace{0.6em}

\begin{proposition}[Repersenter Theorem for Nyström subsampling]
   Let $\mathbf{K}$ denote the Gram matrix associated with the scalar-valued kernel $k$, where $\mathbf{K} = \bigl(k(x_i, x_j)\bigr)_{i,j=1}^{n}$, and
    \[
        K_{n,m} = \left(k(x_i,\tilde{x_j})\right)_{i,j=1}^{n,m}.
    \]
   Let $\mathbf{y} = (y_1, \ldots, y_n)$ and define the norm $\|\cdot\|_n$ by
\[
\|\mathbf{y}\|_n^2 = \frac{1}{n} \sum_{i=1}^n \|y_i\|_{\mathcal{Y}}^2.
\]
Then the estimator $\widehat{f}_{m,\lambda}$ admits the representation
    \begin{equation}
        \widehat{f}_{m,\lambda}(x)
        \;=\;
 \sum_{j=1}^{m} \tilde b_j\, k_{\tilde{x}_j},
        \qquad
        \mathbf{b} = (\tilde{b}_1, \tilde{b}_2, \ldots, \tilde{b}_m)
        \;=\;
        \left( K^\top_{n,m} K_{n,m} +\lambda  I \right)^{-1} K^\top_{n,m} \mathbf{y}.
        \label{eq:nysrepresenter}
    \end{equation}
\end{proposition}
    
\begin{proof}
\noindent
Let
\[
    T(\mathbf{b}) 
    = 
    \frac{1}{n} \sum_{i=1}^n 
    \big\|  
        y_i - 
       \Big( \sum_{j=1}^m \tilde{b}_j k(\cdot, \tilde{x}_j) \Big) x_i 
    \big\|_\mathcal{Y}^2  ,
\]

\noindent
which can be written as

\[
    T(\mathbf{b}) 
    = 
    \frac{1}{n} \sum_{i=1}^n  
    \big\|  
        y_i -   [K_{n,m}\,\mathbf{b}]_i 
    \big\|_\mathcal{Y}^2.
\]

\noindent
Consider
\begin{align*}
    T(\mathbf{b+h}) - T(\mathbf{b})
    &= 
    \frac{1}{n} \sum_{i=1}^n  
    \big\|  
        y_i -   [K_{n,m}(\mathbf{b+h})]_i 
    \big\|_\mathcal{Y}^2
    -
    \frac{1}{n} \sum_{i=1}^n  
    \big\|  
        y_i -   [K_{n,m}\mathbf{b}]_i 
    \big\|_\mathcal{Y}^2 \\[4pt]
    &=
    \big\|  
        \mathbf{y} -  [K_{n,m}(\mathbf{b+h})] 
    \big\|_n^2 
    -  
    \big\|  
        \mathbf{y} -  [K_{n,m}\mathbf{b}] 
    \big\|_n^2  \\[4pt] 
    &=
    \big\|   [K_{n,m}\mathbf{h}] 
    \big\|_n^2 
    -  
    2\big\langle  
        \mathbf{y} -  [K_{n,m}\mathbf{b}] ,  [K_{n,m}\mathbf{h}] 
    \big\rangle_n .
\end{align*}

\noindent
Hence, taking the Fréchet derivative of $T(\mathbf{b})$ and setting it equal to zero, we obtain
\[
    \left(K^\top_{n,m}K_{n,m} \right)\mathbf{b} 
    =  
    K^\top_{n,m} \mathbf{y}.
\]

\noindent
Incorporating Tikhonov regularization into the above expression gives the desired result.

\end{proof}

\begin{proposition}[Representer Theorem for Nyström Subsampling]
    Let $\mathbf{K}$ denote the Gram matrix associated with the scalar-valued kernel $k$, where
    $\mathbf{K} = \bigl(k(x_i, x_j)\bigr)_{i,j=1}^{n}$, and
    \[
        K_{n,m} = \left(k(x_i,\tilde{x}_j)\right)_{i,j=1}^{n,m}, 
        \qquad
        K_{m,m} = \left(k(\tilde{x}_i,\tilde{x}_j)\right)_{i,j=1}^{m}.
    \]
    Let $\mathbf{y} = (y_1, \ldots, y_n)$ and define the norm $\|\cdot\|_n$ by
    \[
        \|\mathbf{y}\|_n^2 = \frac{1}{n} \sum_{i=1}^n \|y_i\|_{\mathcal{Y}}^2.
    \]
    Then the estimator $\widehat{f}_{m,\lambda}$ minimizing
    \[
        \frac{1}{n}\sum_{i=1}^{n}\left\|y_i - f(x_i)\right\|_{\mathcal{Y}}^2 
        + \lambda\,\|f\|_{\mathcal{G}}^2
    \]
    over $f \in \mathrm{span}\{k_{\tilde{x}_1},\ldots,k_{\tilde{x}_m}\}$
    admits the representation
    \begin{equation}
        \widehat{f}_{m,\lambda}(x)
        \;=\;
        \sum_{j=1}^{m} \tilde{b}_j\, k_{\tilde{x}_j},
        \qquad
        \mathbf{b} = (\tilde{b}_1, \tilde{b}_2, \ldots, \tilde{b}_m)
        \;=\;
        \left(K^\top_{n,m} K_{n,m} + \lambda\, K_{m,m}\right)^{-1} K^\top_{n,m} \mathbf{y}.
        \label{eq:nysrepresenter_rkhs}
    \end{equation}
\end{proposition}

\begin{proof}
\noindent
Since $f = \sum_{j=1}^{m} \tilde{b}_j\, k_{\tilde{x}_j}$, the reproducing property of $\mathcal{G}$ gives
\[
    \|f\|_{\mathcal{G}}^2
    =
    \left\langle
        \sum_{j=1}^{m} \tilde{b}_j\, k_{\tilde{x}_j},\,
        \sum_{l=1}^{m} \tilde{b}_l\, k_{\tilde{x}_l}
    \right\rangle_{\mathcal{G}}
    =
    \sum_{j,l=1}^{m} \tilde{b}_j\,\tilde{b}_l\, k(\tilde{x}_j, \tilde{x}_l)
    =
    \mathbf{b}^\top K_{m,m}\, \mathbf{b}.
\]

\noindent
Define the regularized objective
\[
    T(\mathbf{b})
    =
    \frac{1}{n} \sum_{i=1}^n
    \big\|
        y_i - [K_{n,m}\,\mathbf{b}]_i
    \big\|_{\mathcal{Y}}^2
    +
    \lambda\, \mathbf{b}^\top K_{m,m}\, \mathbf{b},
\]
which can be written as
\[
    T(\mathbf{b})
    =
    \big\|\mathbf{y} - K_{n,m}\,\mathbf{b}\big\|_n^2
    +
    \lambda\, \mathbf{b}^\top K_{m,m}\, \mathbf{b}.
\]

\noindent
Consider
\begin{align*}
    T(\mathbf{b+h}) - T(\mathbf{b})
    &=
    \big\|\mathbf{y} - K_{n,m}(\mathbf{b+h})\big\|_n^2
    +
    \lambda\,(\mathbf{b+h})^\top K_{m,m}(\mathbf{b+h})
    \\
    &\quad
    -\,
    \big\|\mathbf{y} - K_{n,m}\,\mathbf{b}\big\|_n^2
    -
    \lambda\,\mathbf{b}^\top K_{m,m}\,\mathbf{b}.
\end{align*}

\noindent
Expanding the term,
\[
    \big\|\mathbf{y} - K_{n,m}(\mathbf{b+h})\big\|_n^2
    -
    \big\|\mathbf{y} - K_{n,m}\,\mathbf{b}\big\|_n^2
    =
    \big\|K_{n,m}\mathbf{h}\big\|_n^2
    -
    2\big\langle \mathbf{y} - K_{n,m}\,\mathbf{b},\, K_{n,m}\mathbf{h}\big\rangle_n.
\]

\noindent
Expanding the regularization term, using symmetry of $K_{m,m}$,
\[
    (\mathbf{b+h})^\top K_{m,m}(\mathbf{b+h}) - \mathbf{b}^\top K_{m,m}\mathbf{b}
    =
    2\,\mathbf{b}^\top K_{m,m}\mathbf{h}
    +
    \mathbf{h}^\top K_{m,m}\mathbf{h}.
\]

\noindent
Hence,
\[
    T(\mathbf{b+h}) - T(\mathbf{b})
    =
    \big\|K_{n,m}\mathbf{h}\big\|_n^2
    +
    \lambda\,\mathbf{h}^\top K_{m,m}\mathbf{h}
    -
    2\big\langle \mathbf{y} - K_{n,m}\,\mathbf{b},\, K_{n,m}\mathbf{h}\big\rangle_n
    +
    2\lambda\,\mathbf{b}^\top K_{m,m}\mathbf{h}.
\]

\noindent
The terms $\|K_{n,m}\mathbf{h}\|_n^2$ and $\lambda\,\mathbf{h}^\top K_{m,m}\mathbf{h}$ are both $O(\|\mathbf{h}\|^2)$.
Taking the Fréchet derivative (the linear part in $\mathbf{h}$) and setting it equal to zero gives
\[
    K_{n,m}^\top\!\left(\mathbf{y} - K_{n,m}\,\mathbf{b}\right)
    =
    \lambda\, K_{m,m}\,\mathbf{b},
\]
which rearranges to
\[
    \left(K_{n,m}^\top K_{n,m} + \lambda\, K_{m,m}\right)\mathbf{b}
    =
    K_{n,m}^\top\,\mathbf{y}.
\]

\end{proof}

\section{CONVERGENCE ANALYSIS OF OPERATOR LEARNING VIA NYSTR\"OM SUBSAMPLING}\label{section:3}

\noindent
In this section, we provide a theoretical analysis of the proposed framework for Tikhonov regularization under a general source condition. Before stating the key assumptions and lemmas of our analysis, we define the covariance operator
$C_\mathcal X : \mathcal{H} \to \mathcal{H}$ by
\[
C_\mathcal X
:= \mathbb{E}\!\left[(\phi(x)\otimes\phi(x))\right],
\]
where the expectation is taken with respect to the
marginal distribution $\rho_\mathcal X$.

\begin{assumption}\label{kernel-bound}
Let $k :\mathcal  X \times\mathcal  X \to \mathbb{R}$ be the reproducing kernel of a scalar-valued RKHS $\mathcal{H}$, satisfying
\[
k(x,x) \le \kappa^2 \quad \mathrm{for\ all}\ x \in \mathcal X .
\]
\end{assumption}

\begin{assumption}\label{bound-Y-asssumption}
There exists a constant $M' > 0$ such that,  $\|y\|_\mathcal Y \le M'$ for all $y \in\mathcal  Y$.

\end{assumption}

To analyze error bounds and establish faster learning rates for learning algorithms, it is essential to exploit the smoothness of the target function, typically characterized via a source condition. In this work, we consider a general source condition, as considered in \cite{holzleitner2026regularizedlearningfunctionaldata}, which extends and generalizes the classical Hölder-type source condition.

{

\begin{assumption}{(General source condition)}\label{assump:source}
{
We assume that the operator $C_\ast$ belongs to the Hilbert--Schmidt class
$\mathcal S_2(\mathcal H^{\varphi},\mathcal Y)$, where $\varphi$ is the index function and  the space $\mathcal H^{\varphi}$ is defined as
\[
\mathcal H^{\varphi}
:= \left\{ f \in \mathcal H \;:\; f = \varphi(C_\mathcal X) g \quad \mathrm{for \, some } \; g \in \mathcal H \right\}.
\]
}
\end{assumption}

\noindent
{
Note that the space $\mathcal H^{\varphi}$ is a Hilbert space when endowed with the inner product
\[
\big\langle \varphi(C_\mathcal X) g_1,\, \varphi(C_\mathcal X) g_2 \big\rangle_{\mathcal H^{\varphi}}
:= \langle g_1, g_2 \rangle_{\mathcal H},
\qquad g_1, g_2 \in \mathcal H.
\]
}
{
Consequently, the family $
\{\varphi(\mu_i) f_i\}_{i \ge 1}$
forms an orthonormal basis of $\mathcal H^{\varphi}$.
}

\begin{definition}\label{qualfication_covering}
We say that the qualification of Tikhonov regularization  \emph{covers} the index function $\varphi$, if there exists a constant $c>0$ such that
\begin{equation}\label{eq:cover}
c\,\frac{t^{}}{\varphi(t)}
\;\le\;
\inf_{\,t \le \sigma \le s}
\frac{\sigma}{\varphi(\sigma)},
\end{equation}
for all $0 < t \le s$.
Here, $\varphi$ is an index function, i.e., $\varphi:[0,\infty)\to[0,\infty)$ is continuous, strictly increasing, and satisfies $\varphi(0)=0$.
\end{definition}

\begin{assumption}\label{eigendecay}
Let $\{(\mu_i,f_i)\}_{i\in\mathbb{N}}$ be the eigenpairs of $C_\mathcal X$ with strictly positive eigenvalues $\mu_i>0$, and corresponding eigenfunctions $\{f_i\}_{i\in\mathbb{N}}$ forming an orthonormal basis of $\mathcal H$. Assume that for some $\mathbf{b} > 1$, the eigenvalues satisfy
\[
i^{-\mathbf{b}} \;\lesssim\; \mu_i \;\lesssim\; i^{-\mathbf{b}},
\qquad \forall\, i \in \mathbb{N}.
\]
\end{assumption}

\noindent
This assumption characterizes the polynomial decay of the integral operator's eigenvalues, which in turn allows us to control the \emph{effective dimension} 
\[
\mathcal{N}(\lambda) 
:= 
{Tr}\!\left(C_\mathcal X\big(C_\mathcal X+ \lambda I\big)^{-1}
\right),
\]
which, under the above condition, admits the bound
\begin{center}
$\mathcal{N}(\lambda) \;\lesssim\; \lambda^{-\frac{1}{\mathbf{b}}}.
$
\end{center}

\begin{lemma}[\cite{Pinelis1985,yurinsky1995sums}]\label{PINELIS:INEQUALITY} 
Let $(\Omega,\mathcal{F},\mathbb{P})$ be a probability space and let $\xi$ be a random variable on $\Omega$ taking values in a real separable Hilbert space $\mathcal{H}$. 
Assume that there exist two positive constants $L$ and $\sigma$ such that

\begin{equation}\label{moment-equation}
\mathbb{E}\!\left[\big\|\xi-\mathbb{E}[\xi]\big\|_{\mathcal{H}}^{m}\right]
\le \frac{1}{2} m!\,\sigma^{2} L^{m-2},
\qquad \forall\, m \ge 2 .
\end{equation}
Let $\xi_1,\ldots,\xi_n$ be i.i.d.\ copies of $\xi$. Then, for all $n \in \mathbb{N}$ and $0<\delta<1$, it holds that

\[
\mathbb{P}
\left(
\left\|
\frac{1}{n}\sum_{i=1}^{n} \xi_i - \mathbb{E}[\xi]
\right\|_{\mathcal{H}}
\le
2\left(\frac{L}{n} + \frac{\sigma}{\sqrt{n}}\right)
\log \left(\frac{2}{\delta}\right)
\right)
\ge 1-\delta .
\]
Moreover, the moment condition~(\ref{moment-equation}) holds provided that

\[
\|\xi(\omega)\|_{\mathcal{H}} \le \frac{L}{2} \quad {almost \,surely},
\quad {and} \quad
\mathbb{E}\!\left[\|\xi\|_{\mathcal{H}}^{2}\right] \le \sigma^{2}.
\]
\end{lemma}

\begin{lemma}\label{opeator-norm-lemma}
With probability at least $1-\delta$, the following bound holds

\[
\bigl\|
( \widehat{C}_{\mathcal Y,\mathcal X}-{C_*}\widehat C_\mathcal X)\,
(C_\mathcal X+\lambda I)^{-1/2}
\bigr\|_{\mathcal  S_2(\mathcal H, \mathcal Y)}
\;\leq\;
\frac{4M\kappa}{n\sqrt{\lambda}}
\log\!\left( \frac{4}{\delta} \right)
+ 2M
\sqrt{
\frac{\mathcal{N}(\lambda)}{n}
}\log\!\left( \frac{4}{\delta} \right).
\]

\end{lemma}

\begin{proof}
    
To bound the above term, we define the random variable
\[
\xi : \mathcal X \times \mathcal  Y \to \mathcal S_2(\mathcal H, \mathcal Y), 
\qquad
\xi(x,y)
=
\bigl(y-C_*\phi(x)\bigr)
\otimes
(C_\mathcal X+\lambda I)^{-1/2}\phi(x).
\]
Let $M=\|y\|_\mathcal Y+\|C_*\|_{\mathcal S_2(\mathcal H, \mathcal Y)}\kappa$, then
\[
\|\xi(x,y)\|_{\mathcal S_2(\mathcal H, \mathcal Y)}
\leq
\frac{\|y\|_\mathcal Y+\|C_*\|_{\mathcal S_2(\mathcal H, \mathcal Y)}\kappa}{\sqrt{\lambda}}
\,\kappa =
\frac{M\,\kappa}{\sqrt{\lambda}}
 .
\]

\noindent
Moreover,
\[
\mathbb{E}\bigl[\|\xi\|_{\mathcal S_2(\mathcal H, \mathcal Y)}^2\bigr]\leq
(\|y\|_\mathcal Y+\|C_*\|_{\mathcal S_2(\mathcal H, \mathcal Y)}\kappa)^2
\int_{\mathcal X}
\bigl\|
(C_\mathcal X+\lambda I)^{-1/2}\phi(x)
\bigr\|_{\mathcal H}^2
\,d\rho_X(x).
\]

\noindent
Using Lemma~17 of \cite{meunier2024optimal}, we obtain
\[
\mathbb{E}\bigl[\|\xi\|_{\mathcal S_2(\mathcal H, \mathcal Y)}^2\bigr]
\leq
M^2
\mathcal{N}(\lambda).
\]

\noindent
Applying Lemma~\ref{PINELIS:INEQUALITY}, we obtain that with probability at least $1-\delta$,
\begin{align*}
\bigl\|
( \widehat{C}_{\mathcal Y,\mathcal X}-{C_*}\widehat C_\mathcal X)\,
(C_\mathcal X+\lambda I)^{-1/2}
\bigr\|_{\mathcal  S_2(\mathcal H, \mathcal Y)}
\;\leq\;
\frac{4M\kappa}{n\sqrt{\lambda}}
\log\!\left( \frac{4}{\delta} \right)
+ 2M
\sqrt{
\frac{\mathcal{N}(\lambda)}{n}
}\log\!\left( \frac{4}{\delta} \right).
\end{align*}

\end{proof}
Using the key lemmas and assumptions introduced above, we now establish the error bound of the proposed algorithm in $\mathcal{S}_2(L^2, \mathcal{Y})$ norm.

\begin{theorem}
Suppose that {assumptions~\ref{kernel-bound}--\ref{assump:source}} are satisfied.
If  the qualification of Tikhonov regularization  covers the index function \(\varphi(t)\sqrt{t}\),
then with probability at least \(1-\delta\), we have,
\[
\|C_* - \widehat{C}_\lambda^P\|_{\mathcal{S}_2(L^2,\mathcal Y)}
\;\le\;8
\|C_o\|_{\mathcal S_2(\mathcal H, \mathcal Y)}
\sqrt{\lambda}\,\varphi(\lambda)+4\left(
\frac{2M\kappa}{n\sqrt{\lambda}}
\log\!\left(\frac{4}{\delta}\right)
+
M
\sqrt{\frac{\mathcal N(\lambda)}{n}}
\log\!\left(\frac{4}{\delta}\right)
\right).
\]
\end{theorem}

\begin{proof}
We decompose the error as:
\begin{align*}
\|C_* - \widehat{C}_\lambda^P\|_{\mathcal{S}_2(L^2,\mathcal Y)}
&\le 
\underbrace{\|C_* - C_* \widehat{C}_{\mathcal X} P g_\lambda(P \widehat{C}_{\mathcal X} P)\|_{\mathcal{S}_2(L^2,\mathcal Y)}}_{I_1}+
\underbrace{\|C_* \widehat{C}_{\mathcal X} P g_\lambda(P \widehat{C}_{\mathcal X} P) - \widehat{C}_\lambda^P\|_{\mathcal{S}_2(L^2,\mathcal Y)}}_{I_2}.
\end{align*}

\noindent
\underline{\emph{Term-$I_{1}$:}}
\begin{align*}
I_1
&\le
\|C_*(I-P)\|_{\mathcal{S}_2(L^2,\mathcal Y)}
+
\|C_*(P-\widehat{C}_{\mathcal X} P g_\lambda(P\widehat{C}_{\mathcal X} P))\|_{\mathcal{S}_2(L^2,\mathcal Y)} \\
&\le \underbrace{
\|C_*(I-P)(C_\mathcal X+\lambda I)^{1/2}\|_{\mathcal{S}_2(\mathcal H, \mathcal Y)}}_{I_{1,1}}
+\underbrace{
\|C_*(P-\widehat{C}_{\mathcal X} P (P\widehat{C}_{\mathcal X} P+\lambda I)^{-1})(C_\mathcal X+\lambda I)^{1/2}\|_{\mathcal{S}_2(\mathcal H, \mathcal Y)}}_{I_{1,2}} .
\end{align*}

\noindent
\underline{\emph{Bound for term-$I_{1,1}$:}}
\begin{align*}
& \|C_*(I-P)(\widehat C_\mathcal X+\lambda I)^{1/2}\|_{\mathcal{S}_2(\mathcal H, \mathcal Y)}\\
\le &
\|C_*(\widehat C_\mathcal X+\lambda I)^{-1/2}\|_{\mathcal S_2(\mathcal H, \mathcal Y)}
\|
(\widehat C_\mathcal X+\lambda I)^{1/2}(I-P)(\widehat C_\mathcal X+\lambda I)^{1/2}
\|_{op}\\
\leq & \|C_o\|_{\mathcal S_2(\mathcal H, \mathcal Y)}
\|
\varphi(C_\mathcal X)(\widehat C_\mathcal X+\lambda I)^{-1/2}\|_{op}
\|
(\widehat C_\mathcal X+\lambda I)^{1/2}(I-P)(\widehat C_\mathcal X+\lambda I)^{1/2}
\|_{op}\\
\stackrel{(*)}{\leq} & 3\lambda \|C_o\|_{\mathcal S_2(\mathcal H, \mathcal Y)}
\|
\varphi(C_\mathcal X)(\widehat C_\mathcal X+\lambda I)^{-1/2}\|_{op}\\
\leq & 3\|C_o\|_{\mathcal S_2(\mathcal H, \mathcal Y)}\sqrt{\lambda}\,\varphi(\lambda),
\end{align*}
where $(*)$ follows from \cite[Lemma 6]{less_is_more} and for the last step we use \cite[Lemma 2]{Gupta_2025}.\\

\noindent
\underline{\emph{Bound for term-$I_{1,2}$:}}
\begin{align*}
    & \|C_*(P-\widehat{C}_{\mathcal X} P (P\widehat{C}_{\mathcal X} P+\lambda I)^{-1})(C_\mathcal X+\lambda I)^{1/2}\|_{\mathcal{S}_2(\mathcal H, \mathcal Y)}\\
    \leq & \|C_*P(I-\widehat C_\mathcal XP(P\widehat C_\mathcal XP+\lambda I)^{-1})(C_\mathcal X+\lambda I)^{1/2}\|_{\mathcal S_2(\mathcal{H}, \mathcal Y)} \\
&\quad+
\|C_*(I-P)\widehat C_\mathcal XP(P\widehat C_\mathcal XP+\lambda I)^{-1}(C_\mathcal X+\lambda I)^{1/2}\|_{\mathcal S_2(\mathcal{H}, \mathcal Y)}\\
\leq & \|C_*P(I-\widehat C_\mathcal XP(P\widehat C_\mathcal XP+\lambda I)^{-1})(C_\mathcal X+\lambda I)^{1/2}\|_{\mathcal S_2(\mathcal{H}, \mathcal Y)} + 3\|C_o\|_{\mathcal S_2(\mathcal H, \mathcal Y)}\sqrt{\lambda}\varphi(\lambda)\\
\stackrel{(*)}{=} & \lambda
\|C_*P(P\widehat C_\mathcal XP+\lambda I)^{-1}(C_\mathcal X+\lambda I)^{1/2}\|_{\mathcal S_2(\mathcal{H}, \mathcal Y)} + 3\|C_o\|_{\mathcal S_2(\mathcal H, \mathcal Y)}\sqrt{\lambda}\varphi(\lambda)\\
\leq & 2\|C_o\|_{\mathcal S_2(\mathcal H, \mathcal Y)}
\varphi(\lambda) \sqrt{\lambda} + 3\|C_o\|_{\mathcal S_2(\mathcal H, \mathcal Y)}\sqrt{\lambda}\varphi(\lambda),
\end{align*}
where $(*)$ follows from the fact that $I-\widehat C_\mathcal XP(P\widehat C_\mathcal XP+\lambda I)^{-1}=
\lambda(P\widehat C_\mathcal XP+\lambda I)^{-1}$ and for the last step we use bound from \cite{less_is_more} and \cite[Lemma~5]{Gupta_2025}.\\

\noindent
Combining the above estimates yields
\begin{align*}
I_1
\le
8
\|C_o\|_{\mathcal S_2(\mathcal H, \mathcal Y)}
\sqrt{\lambda}\,\varphi(\lambda).
\end{align*}

\noindent
\underline{\emph{Term-$I_{2}$:}}

\begin{align*}
&\|(C_* \widehat C_\mathcal X-\widehat C_{\mathcal Y, \mathcal X})
P g_\lambda(\widehat C_\mathcal X^P)
(C_\mathcal X+\lambda I)^{1/2}\|_{\mathcal S_2(\mathcal H, \mathcal Y)}\\
& \leq \|(C_*\widehat C_\mathcal X-\widehat C_{\mathcal Y, \mathcal X})(\widehat C_\mathcal X+\lambda I)^{-1/2}\|_{\mathcal S_2(\mathcal{H},\mathcal Y)}
\|
(\widehat C_\mathcal X+\lambda I)^{1/2}
P g_\lambda(\widehat C_\mathcal X^P)
(C_\mathcal X+\lambda I)^{1/2}
\|_{op}.
\end{align*}

\noindent
The bound for the first term is given in Lemma \ref{opeator-norm-lemma}. The bound for the second term is 1, as established by Theorem 1 in \cite{shuai_lu}.

Thus
\begin{align*}
I_2
\le
2\left(
\frac{4M\kappa}{n\sqrt{\lambda}}
\log\!\left(\frac{4}{\delta}\right)
+
2M
\sqrt{\frac{\mathcal N(\lambda)}{n}}
\log\!\left(\frac{4}{\delta}\right)
\right).
\end{align*}

\noindent
Finally, combining the bounds for \(I_1\) and \(I_2\) and applying a union bound yields the desired result with probability at least \(1-\delta\).
\end{proof}

The previous result guarantees a convergence rate for the proposed approximation scheme. Under an additional Assumption~\ref{eigendecay}, the next theorem establishes the corresponding optimal convergence rate. The proof is analogous to that of the previous theorem and is omitted for brevity.


\begin{theorem}\label{Rates-theorem-general}
Suppose that {assumptions~\ref{kernel-bound}--\ref{eigendecay}} hold and  $ \delta \in (0,  1)$. 
Assume $\lambda \in (0,1]$ satisfies $\lambda = \Psi^{-1}(n^{-\frac{1}{2}})$, 
where $\Psi(t) = t^{\frac{1}{2} + \frac{1}{2\mathbf{b}}} \varphi(t)$. If the index function $\varphi(t)\sqrt{t}$ is covered by the qualification of the Tikhonov regularization, then with probability at least $1-\delta$, we have
\[
\|C_* - \widehat{C}_\lambda^P\|_{\mathcal{S}_2(L^2,\mathcal Y)}
\leq
c\,
\left(\Psi^{-1}\!\left(n^{-\frac{1}{2}}\right)\right)^{\frac{1}{2}}
\varphi\!\left(\Psi^{-1}\!\left(n^{-\frac{1}{2}}\right)\right)
\log\!\left(\frac{2}{\delta}\right),
\]
for some positive constant $c$.

\end{theorem}

\begin{corollary}\label{Holder-corollary}
Suppose that Assumptions~\ref{kernel-bound}--\ref{eigendecay} hold and let
$\delta \in (0,1)$. 
Assume that the Hölder's source condition holds, i.e.,
$\varphi(t)=t^r$ for some $r>0$,
and choose the regularization parameter
$\lambda = n^{-\frac{\mathbf b}{2\mathbf b r+\mathbf b+1}}$.
Assume that $\varphi(t)\sqrt{t}$ is covered by the qualification of Tikhonov regularization. Then, for any $0 < r < \tfrac{1}{2}$, it holds with probability at least $1 - \delta$ that
\[
\|C_* - \widehat{C}_\lambda^P\|_{\mathcal{S}_2(L^2,\mathcal Y)}
\le
c\,
n^{-\frac{2 \textbf{b} r+\textbf{b}}{4 \textbf{b} r+2 \textbf{b}+2}}
\log\!\left(\frac{2}{\delta}\right),
\]
for some positive constant $c$.

\end{corollary}

\begin{remark}
Under the polynomial eigenvalue decay assumption, the convergence rates in
Theorem~\ref{Rates-theorem-general} and Corollary~\ref{Holder-corollary}
coincide with the minimax optimal rates reported in \cite{rastogi2017optimal}
and \cite{less_is_more} for the general and Hölder source conditions,
respectively. 
Moreover, even without Assumption~\ref{eigendecay}, i.e., in the absence of
polynomial eigenvalue decay, estimates analogous to those in Theorem~10 of
\cite{regularization} can still be established. 
Hence, under the considered assumptions, the proposed estimator achieves
optimal convergence rates.
\end{remark}

\section{NUMERICAL RESULTS}\label{section:4}

This section presents numerical experiments evaluating the proposed approach 
across three core denoising tasks --- signal denoising, real-time signal 
denoising, and image denoising --- as well as two additional problems: CT 
image reconstruction from sinogram measurements and a low-dimensional 
regression problem for energy-efficiency prediction. Across all experiments, 
we use a Gaussian kernel of the form
\begin{equation}
    k(x, t) = e^{-\gamma \|x - t\|^2},
\end{equation}
where $\gamma > 0$ is a bandwidth parameter controlling the rate of decay, 
and investigate the effectiveness of the proposed Nystr\"{o}m-based 
approximation, demonstrating that substantial reductions in computational 
cost can be achieved while maintaining performance comparable to the full 
kernel method.

In the signal denoising experiments, we assess the ability of the method to recover clean signals from noisy observations, examining how well it preserves important signal features while suppressing noise. The real-time denoising experiments target streaming settings, where denoised outputs must be produced as new measurements arrive; here we also benchmark the proposed framework against several existing denoising techniques. The image denoising experiments consider recovery from both additive noise and motion blur — two degradation types commonly encountered in practice. The CT reconstruction experiment addresses the classical inverse problem of recovering an image from tomographic projection data, while the regression experiment estimates heating and cooling loads from building design characteristics, illustrating applicability to high-dimensional and functional outputs alike.

All experiments were conducted on a MacBook Air with an Apple M4 processor (10-core CPU, 8-core GPU, 16~GB unified memory) running macOS Tahoe. The implementation is available.\footnote{\url{https://anonymous.4open.science/r/Nystrom_experiments-013F/README.md}}

\subsection{Signal Denoising}
\subsubsection{Offline Signal Denoising}

We begin with a synthetic signal denoising experiment. For each realization, the clean signal is generated as a superposition of two sinusoidal components. Let
\[
t_n = \frac{n}{f_s}, \qquad n=0,1,\ldots,N-1,
\]
where \(f_s=1000\) Hz denotes the sampling frequency and \(N\) is the signal length. Let $\mathcal{U}$ denotes the uniform distribution, then the clean signal is defined by
\[
x_n =
a \sin\left(2\pi p t_n + \phi_1\right)
+
b \sin\left(2\pi q t_n + \phi_2\right),
\]
where the frequencies, amplitudes, and phases are sampled independently according to
\[
p \sim \mathcal{U}(30,90), \qquad
q \sim \mathcal{U}(100,220),
\]
\[
a \sim \mathcal{U}(0.8,1.2), \qquad
b \sim \mathcal{U}(0.3,0.7),
\]
\[
\phi_1,\phi_2 \sim \mathcal{U}(0,2\pi).
\]
This construction produces signals with randomized spectral content, amplitudes, and phase offsets. Each signal has total length $16{,}384$ samples and is divided into $128$ non-overlapping segments of length $128$ samples each.

To simulate noisy observations, Gaussian noise $\varepsilon \sim \mathcal{N}(0,\sigma^2)$ is added to each clean signal, with standard deviation $\sigma_{\text{train}} = 0.9$ for training and $\sigma_{\text{test}} = 1.5$ for testing. In the experiment, 20 distinct clean signals are used for training, yielding $20 \times 128 = 2{,}560$ segments in total, which are partitioned into an $80$--$20$ training-validation split. Testing is performed on 100 newly generated random signals as mentioned above, so the model must generalize across both noise realizations and different underlying signal structures.

We briefly describe our signal denoising method. Let $x_i \in \mathbb{R}^{128}$ denote the $i$-th noisy segment and $y_i \in \mathbb{R}^{128}$ the corresponding clean segment. We compute the Discrete Cosine Transform (DCT) of each noisy segment to obtain the feature representation
\begin{equation}
    \tilde{x}_i = \mathcal{D}(x_i) \in \mathbb{R}^{128},
\end{equation}
where $\mathcal{D} : \mathbb{R}^{128} \to \mathbb{R}^{128}$ denotes the DCT operator. The goal of our operator learning algorithm is then to approximate the map
\begin{equation}
    \mathcal{F} : \tilde{x}_i \mapsto y_i,
\end{equation}
recovering the clean signal $y_i$ from the noisy observations $\tilde{x}_i$. Hyperparameters $\gamma$ and $\lambda$ are jointly selected over the grids $\gamma \in [10^{-6}, 10^{1}]$ and $\lambda \in [10^{-6}, 10^{-1}]$, each with $20$ log-spaced values, using the validation split described above.

We compare the full KRR solution against its Nystr\"{o}m approximation for 
$m \in \{50, 75, 100, 125, 150\}$. The Nystr\"{o}m method selects $m$ subsample points uniformly at random from the training set and constructs a low-rank approximation of the kernel matrix, replacing the $\mathcal{O}(n^2)$ kernel system with an $m \times m$ eigendecomposition.

Performance is evaluated using the mean squared error (MSE) and the signal-to-noise ratio (SNR) gain (in dB), together with prediction time per segment to assess computational efficiency. Each experiment is repeated $100$ times with different random seeds, and results are reported as means and standard deviations across repetitions. Table~\ref{tab:offline} summarizes the quantitative results.

\begin{table}[h]
\centering
\caption{Comparison of Full KRR and Nyström KRR for varying $m$ in terms of MSE, SNR Gain, and computation time}
\label{tab:offline}
\begin{tabular}{lcccc}
\toprule
\textbf{Method} & \textbf{$m$} & \textbf{MSE} & \textbf{SNR Gain (dB)} & \textbf{Time (ms)} \\
\midrule
Full KRR      & -- & 0.4699 & 6.823 & 0.0111  \\
\midrule
Nyström KRR   & 50  & 0.6363 & 5.506 & 0.000727 \\
Nyström KRR   & 75  & 0.6211 & 5.613 & 0.001080 \\
Nyström KRR   & 100 & 0.5798 & 5.912 & 0.001384 \\
Nyström KRR   & 125 & 0.4888 & 6.644 & 0.001787 \\
Nyström KRR   & 150 & 0.4663 & 6.847 & 0.002232 \\
\bottomrule
\end{tabular}
\end{table}




\subsubsection{Real-Time Signal Denoising}
We conduct experiments on synthetic audio signals corrupted by additive Gaussian 
noise, with the objective of recovering clean audio from noisy observations in a 
streaming setting. The audio signals are generated at a sampling frequency of 
$f_s = 16{,}000$ Hz with a duration of $5$ seconds, yielding $80{,}000$ samples 
per signal. Each signal is constructed as a sequence of six sinusoidal tones, 
where each tone is a harmonic triplet of the form 
$\sin(2\pi f t) + 0.5\sin(4\pi f t) + 0.3\sin(6\pi f t)$. The resulting segments 
are concatenated and the composite signal is normalized to the range 
$[-1, 1]$.

Frequency sets are constructed by selecting a base frequency $f_0$ and scaling 
it by fixed harmonic factors $[1.0,\, 1.5,\, 2.0,\, 2.5,\, 3.0,\, 2.0]$, 
yielding six-note sequences with a consistent harmonic structure. Training uses 
five such note sets with base frequencies ranging from $220$ Hz to $330$ Hz, 
and validation uses two additional sets. For testing, $100$ unseen note sets are 
generated by sampling base frequencies uniformly from $[250,\, 360]$ Hz, 
producing frequency combinations not seen during training. Gaussian noise with standard deviation 
$\sigma_{\text{train}} = 0.9$ is added to the training and validation signals, 
while $\sigma_{\text{test}} = 1.5$ is used for the test signals, introducing a 
deliberate train/test noise mismatch that makes the task particularly challenging.

Let $x_i \in \mathbb{R}^{256}$ and $y_i \in \mathbb{R}^{256}$ denote the $i$-th 
noisy and clean frame, respectively, extracted with a Hann window of frame size 
$256$ and hop size $ 64$. We compute the DCT of each frame as
\begin{equation}
    \tilde{x}_i = \mathcal{D}(x_i) \in \mathbb{R}^{256},
\end{equation}
where $\mathcal{D} : \mathbb{R}^{256} \to \mathbb{R}^{256}$ denotes the DCT 
operator. The input feature is formed by concatenating the DCT coefficients of 
the previous and current noisy frames,
\begin{equation}
    z_i = [\tilde{x}_{i-1},\,\tilde{x}_i] \in \mathbb{R}^{512},
\end{equation}
and the regression target is the residual between the clean and noisy spectral 
coefficients,
\begin{equation}
    r_i = \mathcal{D}(y_i) - \tilde{x}_i \in \mathbb{R}^{256}.
\end{equation}
The operator learning algorithm approximates the map 
$\mathcal{F} : z_i \mapsto r_i$, 
from which the clean frame is recovered as 
$\hat{y}_i = \mathcal{D}^{-1}(\tilde{x}_i + \mathcal{F}(z_i))$.

Each signal of $80{,}000$ samples yields $\lfloor 80{,}000 / 64 \rfloor = 1{,}250$ 
frames, so the five training signals produce a total of $N = 6{,}250$ training 
frames, giving
\[
    X_{\text{train}} \in \mathbb{R}^{6250 \times 512}, \qquad 
    Y_{\text{train}} \in \mathbb{R}^{6250 \times 256}.
\]
Model selection is carried out over Gaussian kernel bandwidths 
$\gamma \in \{10^{-3},\,\dots,\,10^{2}\}$ and regularization parameters 
$\lambda \in \{10^{-9},\,\dots,\,10^{-2}\}$, each sampled on a logarithmic grid 
of $10$ values. For the Nystr\"{o}m approximation, the number of subsample points 
is varied over $m \in \{100,\,200,\,300,\,400,\,500\}$. The entire pipeline is repeated over 
$100$ independent trials.

In the real-time implementation, the signal is processed causally: a rolling 
buffer stores only the most recent frame of noisy samples, and a prediction is 
made whenever a hop boundary is reached. Since the input feature combines only 
the current and previous noisy DCT coefficients, the inference rule uses no 
future information and can be implemented in a streaming fashion. With a frame 
size of $256$ samples at $f_s = 16{,}000$ Hz, the analysis context spans
\[
    \frac{256}{16{,}000} = 16~\text{ms},
\]
and the enhanced waveform is updated every
\[
    \frac{64}{16{,}000} = 4~\text{ms}.
\]

Performance is evaluated using MSE and SNR gain between the reconstructed and 
ground-truth clean signals, averaged over the $100$ test signals. The average 
MSE of the noisy test signals before denoising is approximately $2.2453$, while 
the training-set MSE is approximately $0.8108$, reflecting the noise mismatch 
between training and test conditions. Results for different Nystr\"{o}m ranks $m$ 
are reported in Table~\ref{tab:real.time}.
\begin{table}[h]
\centering
\caption{Comparison of Full KRR and Nyström KRR for varying $m$ in terms of MSE, SNR Gain, and computation time}
\label{tab:real.time}
\begin{tabular}{lcccc}
\toprule
\textbf{Method} & \textbf{$m$} & \textbf{MSE} & \textbf{SNR Gain (dB)} & \textbf{Time (s)} \\
\midrule
Full KRR      & -- & 0.2106 & 10.28 & 7.1  \\
\midrule
Nyström KRR   & 100  & 1.0606 & 3.26 & 1.0  \\
Nyström KRR   & 200 & 0.3317 & 8.31 & 1.3 \\
Nyström KRR   & 300 & 0.2313 & 9.89 & 1.9 \\
Nyström KRR   & 400 & 0.2197 & 10.10 & 2.8 \\
Nyström KRR   & 500 & 0.2151 & 10.20 & 3.4 \\
\bottomrule
\end{tabular}
\end{table}

\subsubsection{Comparison with other signal denoising methods}
We compare the denoising performance of the proposed method against well-known 
wavelet-based signal denoising approaches~\cite{real_time_ieee,wavelet_denosiing_2000}. 
We evaluate over a broad collection of wavelet 
families, including \texttt{bior1.1}--\texttt{bior2.6}, 
\texttt{coif1}--\texttt{coif5}, \texttt{db2}--\texttt{db11}, 
\texttt{rbio1.3}--\texttt{rbio2.8}, and \texttt{sym2}--\texttt{sym7}. 
These wavelets span a wide range of filter lengths, symmetry properties, and 
reconstruction characteristics, providing a comprehensive benchmark~\cite{optimal_wavelet}.
Each synthetic signal consists of six sequential tones, each a harmonic 
triplet of the form $\sin(2\pi f t) + 0.5\sin(4\pi f t) + 0.3\sin(6\pi f t)$,  concatenated to form a signal of length 
$N$ sampled at frequency $f_s$. The composite signal is 
corrupted by additive Gaussian noise.

Performance is evaluated on 100 unseen test signals with previously unseen 
frequency combinations. Results are reported in terms of mean squared error (MSE) 
and signal-to-noise ratio (SNR) gain, averaged over 100 independent repetitions 
with additive Gaussian noise of standard deviation $\sigma_{\text{test}} = 1.5$ 
for test samples and $\sigma_{\text{train}} = 0.9$ for training samples.
Table~\ref{tab:wavelet_nystrom_combined} summarizes the results, showing that 
across all wavelet families, the proposed Nystr\"{o}m-based method consistently 
achieves lower MSE and higher SNR gain compared to the classical baseline. 
These results demonstrate that while wavelet thresholding remains a strong and 
competitive baseline, the proposed approach yields consistent improvements across 
both batch and real-time processing regimes.

\begin{table}[!htbp]
\centering
\caption{Denoising performance comparison under offline and online settings.}
\label{tab:wavelet_nystrom_combined}
\begin{tabular}{lcccc}
\toprule
& \multicolumn{2}{c}{\textbf{Offline}} & \multicolumn{2}{c}{\textbf{Real-time}} \\
\cmidrule(lr){2-3} \cmidrule(lr){4-5}
\textbf{Method} & \textbf{MSE} & \textbf{SNR-GAIN} & \textbf{MSE} & \textbf{SNR-GAIN} \\
\midrule
Wavelet-bior1.1 & $0.2327$ & $9.859$ & $0.3976$ & $7.528$ \\
Wavelet-bior1.3 & $0.2460$ & $9.617$ & $0.4313$ & $7.175$ \\
Wavelet-bior1.5 & $0.2549$ & $9.463$ & $0.4627$ & $6.869$ \\
Wavelet-bior2.2 & $0.2508$ & $9.538$ & $0.5230$ & $6.337$ \\
Wavelet-bior2.4 & $0.2301$ & $9.913$ & $0.5264$ & $6.309$ \\
Wavelet-bior2.6 & $0.2282$ & $9.950$ & $0.5481$ & $6.133$ \\
Wavelet-coif1   & $0.2299$ & $9.916$ & $0.4115$ & $7.379$ \\
Wavelet-coif2   & $0.2290$ & $9.935$ & $0.4367$ & $7.121$ \\
Wavelet-coif3   & $0.2289$ & $9.937$ & $0.5567$ & $6.066$ \\
Wavelet-coif4   & $0.2290$ & $9.935$ & $0.5872$ & $5.834$ \\
Wavelet-coif5   & $0.2291$ & $9.935$ & $0.6199$ & $5.598$ \\
Wavelet-db2     & $0.2300$ & $9.913$ & $0.4032$ & $7.468$ \\
Wavelet-db3     & $0.2293$ & $9.928$ & $0.4103$ & $7.392$ \\
Wavelet-db4     & $0.2290$ & $9.935$ & $0.4183$ & $7.308$ \\
Wavelet-db5     & $0.2291$ & $9.933$ & $0.4265$ & $7.224$ \\
Wavelet-db6     & $0.2290$ & $9.936$ & $0.4353$ & $7.135$ \\
Wavelet-db7     & $0.2289$ & $9.937$ & $0.4441$ & $7.048$ \\
Wavelet-db8     & $0.2291$ & $9.935$ & $0.4534$ & $6.958$ \\
Wavelet-db9     & $0.2290$ & $9.936$ & $0.5544$ & $6.084$ \\
Wavelet-db10    & $0.2289$ & $9.937$ & $0.5640$ & $6.009$ \\
Wavelet-db11    & $0.2291$ & $9.934$ & $0.5736$ & $5.936$ \\
Wavelet-rbio1.3 & $0.2420$ & $9.693$ & $0.4206$ & $7.285$ \\
Wavelet-rbio1.5 & $0.2503$ & $9.547$ & $0.4455$ & $7.034$ \\
Wavelet-rbio2.2 & $0.2585$ & $9.401$ & $0.4041$ & $7.458$ \\
Wavelet-rbio2.4 & $0.2318$ & $9.880$ & $0.3931$ & $7.579$ \\
Wavelet-rbio2.6 & $0.2293$ & $9.929$ & $0.4058$ & $7.441$ \\
Wavelet-rbio2.8 & $0.2296$ & $9.924$ & $0.5158$ & $6.397$ \\
Wavelet-sym2    & $0.2300$ & $9.913$ & $0.4032$ & $7.468$ \\
Wavelet-sym3    & $0.2293$ & $9.928$ & $0.4103$ & $7.392$ \\
Wavelet-sym4    & $0.2290$ & $9.934$ & $0.4186$ & $7.305$ \\
Wavelet-sym5    & $0.2291$ & $9.933$ & $0.4268$ & $7.220$ \\
Wavelet-sym6    & $0.2289$ & $9.937$ & $0.4360$ & $7.128$ \\
Wavelet-sym7    & $0.2290$ & $9.936$ & $0.4451$ & $7.038$ \\
\midrule
Nyström-500 
& $\mathbf{0.2188}$ & $\mathbf{10.122}$ 
& $\mathbf{0.2151}$ & $\mathbf{10.1959}$ \\
\bottomrule
\end{tabular}
\end{table}
\vspace{0.3em}

\newpage

Next, we compare our approach with an SVD-based denoising method similar to that proposed in~\cite{Svd_comparison}, adapted to match our experimental pipeline. This method requires two independent noisy observations of the same underlying signal. Accordingly, for each generated signal, we create two noisy realizations by adding independent Gaussian noise with standard deviations $\sigma_1 = 1.5$ and $\sigma_2 = 1.2$. The denoising is performed in a frame-wise manner using overlapping windows with a Hann window and overlap-add reconstruction. For each frame, the average and half-difference of the two noisy observations are computed and embedded into Hankel matrices. Singular value decomposition (SVD) is then applied to both matrices, and the largest singular value of the Hankel matrix corresponding to the difference signal is used as a noise threshold. Singular values of the average Hankel matrix below this threshold are suppressed, and the denoised signal is reconstructed via diagonal averaging (Hankelization) followed by overlap-add synthesis.

We compare this analytical approach with our proposed Nyström-based kernel ridge regression method in an online denoising setting over $100$ independent repartition (experiments), each consisting of $100$ test signals (trials). The noisy signal exhibits an average MSE of $2.255590$. The SVD-based method reduces this error to an average MSE of $0.259554$, corresponding to an average SNR gain of $9.455\,\mathrm{dB}$.  In contrast, our Nyström method (using $500$ subsamples) achieves a lower average MSE of $0.2151$ and a higher average SNR gain of $10.20\,\mathrm{dB}$. These results demonstrate that while the SVD-based method effectively suppresses noise, the proposed Nyström approach consistently achieves superior denoising performance. It is important to note that the SVD-based method relies on the assumption of purely additive noise and the availability of two independent noisy observations; in the presence of convolution noise, its performance degrades significantly, whereas the proposed Nyström method remains robust.

\vspace{0.1em}
\subsection{IMAGE DENOISING}\label{imge_denoising:subsection}

Next, we consider image denoising via operator learning. To simulate realistic degradation, we conduct two independent experiments, each corresponding to a distinct degradation model applied to a face dataset~\cite{bainbridge2013intrinsic}. In both cases, all images are first converted to grayscale and resized to \(100 \times 100\) pixels. The total dataset consists of $10{,}215$ images, out of which $m = 5000$ were selected for Nyström subsampling. Let \(x \in \mathbb{R}^{s \times s}\), with \(s = 100\), denote a clean image. Each experiment yields input--output pairs \(\{(\tilde{x}_i, x_i)\}\), where \(\tilde{x}_i\) is the degraded image and \(x_i\) is the corresponding clean image. We learn a mapping from degraded to clean images using Kernel Ridge Regression (KRR), with hyperparameters \((\gamma, \lambda)\) selected via grid search on a validation set. To contextualize performance, we compare against BM3D~\cite{Bm3d} and DnCNN~\cite{DnCNN-paper}; for a fair comparison, DnCNN is trained and tested on the same samples as our KRR-based approach over 30 epochs.

\subsubsection{Motion Blur}
The degraded observation $\tilde{x}$ is obtained via motion blur:
\begin{equation}
    \tilde{x} = x * k_m,
\end{equation}
where $*$ denotes 2D convolution and $k_m \in \mathbb{R}^{q \times q}$ is a horizontal line kernel whose central row entries are set to one and normalized:
\begin{equation}
k_m(i,j) =
\begin{cases}
\dfrac{1}{q}, & \text{if } i = \lfloor q/2 \rfloor, \\
0, & \text{otherwise.}
\end{cases}
\end{equation}
This ensures $\sum_{i,j} k_m(i,j) = 1$, and we set $q = 9$. The motion blur kernel represents one of the most realistic degradation types encountered in practice~\cite{blurring}, making this experiment a deconvolution problem.

To enrich the feature representation, we compute the 2D Discrete Cosine Transform (DCT) of each degraded image,
\begin{equation}
    C = \mathrm{DCT}_2(\tilde{x}) \in \mathbb{R}^{\sqrt{d} \times \sqrt{d}}, \qquad \sqrt{d} = 100,
\end{equation}
and retain the low-frequency coefficients by truncating $C$ to its top-left block $C_t \in \mathbb{R}^{h \times w}$, where $t = (h, w)$ denotes the truncation threshold. The vectorized coefficients are concatenated with the original pixel values to form a hybrid feature vector:
\begin{equation}
    \phi_t(\tilde{x}) = \begin{bmatrix} \tilde{x} \\ \mathrm{vec}(C_t) \end{bmatrix} 
    \in \mathbb{R}^{d + hw}.
\end{equation}
The low-frequency DCT coefficients capture global structural information, partially compensating for the frequency attenuation introduced by the blur. The threshold $t = (h, w)$ is selected jointly with $(\gamma, \lambda)$ via cross-validation.

Table~\ref{tab:motion_blur_table} summarizes quantitative results in terms of PSNR, SSIM, and LPIPS averaged over 200 unseen test images, and Figure~\ref{Visual_constructions1} provides qualitative reconstructions.

\begin{table}[h]
\centering
\caption{Image reconstruction results under Motion Blur}
\label{tab:motion_blur_table}
\begin{tabular}{lccc}
\hline
Method & PSNR $\uparrow$ & SSIM $\uparrow$ & LPIPS $\downarrow$ \\
\hline
Nystr\"{o}m (ours) & 25.81 & 0.9687 & 0.0962 \\
BM3D               & 20.88 & 0.9538 & 0.3595 \\
DnCNN              & 33.23 & 0.9790 & 0.0095 \\
\hline
\end{tabular}
\end{table}

\begin{figure}[h!]
    \centering
    \includegraphics[width=1\linewidth]{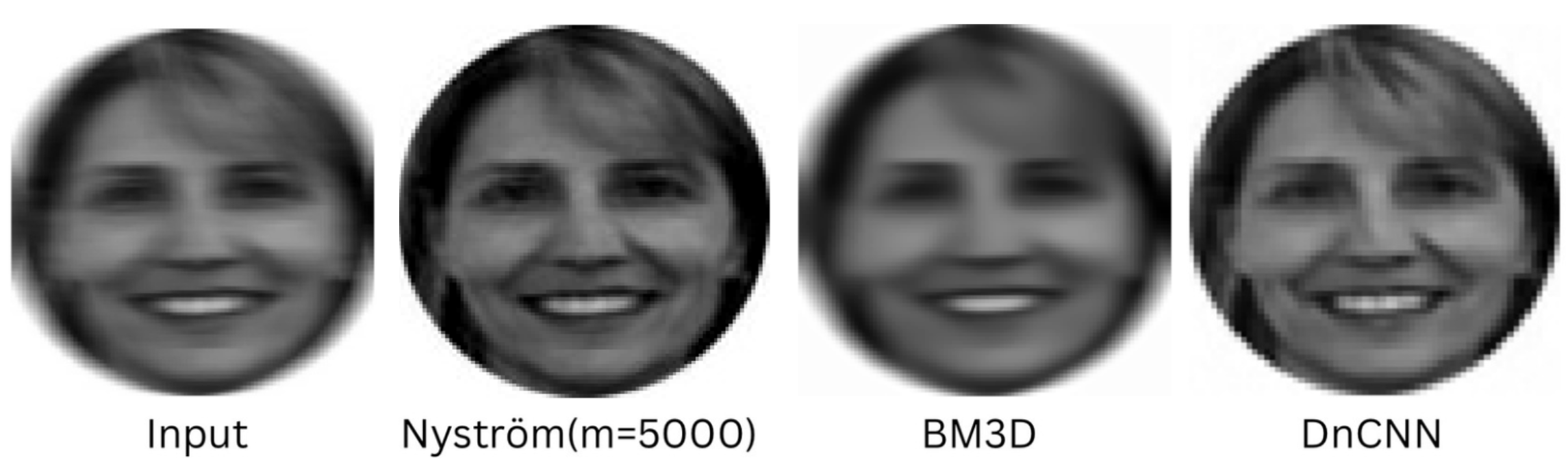}
    \caption{Reconstructed image for KRR, Nystr\"{o}m approximation, BM3D, and DnCNN under motion blur.}
    \label{Visual_constructions1}
\end{figure}

\subsubsection{Gaussian Noise}

The degraded observation is obtained by corrupting the clean image with additive Gaussian noise:
\begin{equation}
    \tilde{x} = x + \eta, \quad \eta \sim \mathcal{N}(0, \sigma^2 I),
\end{equation}
where \(\sigma > 0\) controls the noise level. Unlike the motion blur setting, the inclusion of low-frequency DCT coefficients did not yield significant improvement, so features consist solely of the raw pixel values of \(\tilde{x}\). To assess robustness to variations in noise intensity, evaluations are conducted at a noise level \(\sigma\) different from that used during training. Results in terms of average PSNR, SSIM, and LPIPS over 200 unseen test images are summarized in Table~\ref{tab:gaussian_blur_table}, with qualitative reconstructions in Figure~\ref{Visual_constructions2}.

\begin{table}[h]
\centering
\caption{Image reconstruction results under Gaussian Noise}
\label{tab:gaussian_blur_table}
\begin{tabular}{lccc}
\hline
Method & PSNR $\uparrow$ & SSIM $\uparrow$ & LPIPS $\downarrow$ \\
\hline
Nystr\"{o}m (ours) & 28.62 & 0.9921 & 0.0764 \\
BM3D               & 29.50 & 0.9949 & 0.0910 \\
DnCNN              & 32.48 & 0.9551 & 0.0210 \\
\hline
\end{tabular}
\end{table}

\begin{figure}[h!]
    \centering
    \includegraphics[width=1.\linewidth]{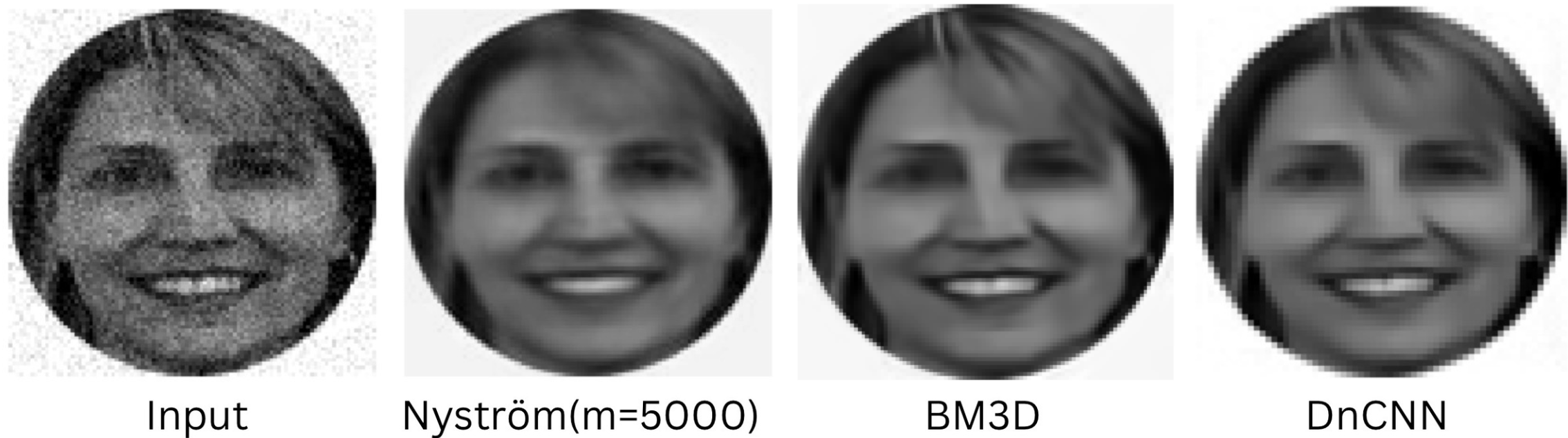}
    \caption{Reconstructed images for KRR, Nystr\"{o}m approximation, BM3D, and DnCNN under Gaussian noise.}
    \label{Visual_constructions2}
\end{figure}
\newpage
We have compared our results with the benchmark denoising methods BM3D and DnCNN available in the literature. The numerical results indicate that the proposed method attains performance on par with these state-of-the-art methods. Furthermore, in some cases, the visual reconstruction results demonstrate the robustness and effectiveness of our operator learning algorithm for image denoising under various blur conditions.

\subsection{Learning the Inverse Radon Transform}
A natural and practically important instance of our operator learning framework arises in computed tomography (CT) image reconstruction, where the underlying operator is given by the Radon transform $\mathcal{R}$, mapping an image $f$ to its sinogram $\mathcal{R}f$ collected over a finite set of projection angles. Rather than inverting this map, we pose CT reconstruction as the problem of \emph{learning the inverse Radon transform directly from data}, treating $\mathcal{R}^{-1}$ as the target operator to be estimated via kernel ridge regression and its Nystr\"om approximation. We use the \emph{Faces} dataset~\cite{bainbridge2013intrinsic}, with each image uniformly resized to $100 \times 100$ pixels and converted to grayscale. For every image $f$, the corresponding sinogram $\mathcal{R}f$ is generated using $70$ projection angles uniformly distributed over $[0^\circ,180^\circ]$ with $142$ detectors per projection, yielding input-output pairs $(\mathcal{R}f, f) \in \mathbb{R}^{9940} \times \mathbb{R}^{10000}$. The learning task is therefore to approximate the map $\mathcal{R}f \mapsto f$, the inverse Radon transform, directly from these paired observations, using a Gaussian radial basis function (RBF) kernel.

Evaluation is performed on $200$ test images, comparing exact kernel ridge regression against its Nystr\"om approximation. Reconstruction quality is quantified using mean squared error (MSE), peak signal-to-noise ratio (PSNR), and relative reconstruction error, together with the average computational time required to predict one sample. The bandwidth $\gamma$ is fixed at $0.01$, and the regularization parameter $\lambda$ is selected via cross-validation. Comparative results for the RBF kernel are summarized in Table~\ref{tab:krr_nystrom_results}, while Figure~\ref{HEATMAP_COMBINED} visualizes the average reconstruction error across Nystr\"om ranks as a heatmap, for both the RBF and inverse multi-quadric (IMQ) kernels.

\begin{table}[ht]
\centering
\caption{Comparison of full KRR and Nyström KRR for $n=10000$. 
Nyström results are shown for various sizes of $m$ for Gaussian kernel.}
\label{tab:krr_nystrom_results}
\begin{tabular}{lccccccc}
\hline
\textbf{Method} & $n$ & $m$  & PSNR & MSE & Rel-L2 & Time (ms) \\
\hline
Full KRR & 10000 & --  & 28.5803 & $1.728{\times}10^{-3}$ & $6.2541{\times}10^{-2}$ & 2.554 \\
Nyström  & 10000 & 50   & 21.6206 & $1.031{\times}10^{-2}$ & $1.5534{\times}10^{-1}$ & 0.051  \\
Nyström  & 10000 & 100   & 22.7393 & $7.826{\times}10^{-3}$ & $1.3495{\times}10^{-1}$ & 0.059 \\
Nyström  & 10000 & 200  & 23.3567 & $5.986{\times}10^{-3}$ & $1.1823{\times}10^{-1}$ & 0.070 \\
Nyström  & 10000 & 500 & 24.1298 & $4.462{\times}10^{-3}$ & $1.0210{\times}10^{-1}$ & 0.128 \\
         & 10000 & 1000  & 24.7846 & $3.836{\times}10^{-3}$ & $9.4721{\times}10^{-2}$ & 0.253 \\
         & 10000 & 1500  & 24.9936 & $3.666{\times}10^{-3}$ & $9.2525{\times}10^{-2}$ & 0.331 \\
         & 10000 & 2000  & 25.2574 & $3.426{\times}10^{-3}$ & $8.9672{\times}10^{-2}$ & 0.506 \\
\hline
\end{tabular}
\end{table}

\vspace{0.4em}
\begin{figure}[h!]
    \centering
    \begin{subfigure}[t]{0.48\linewidth}
        \centering
        \includegraphics[width=\linewidth]{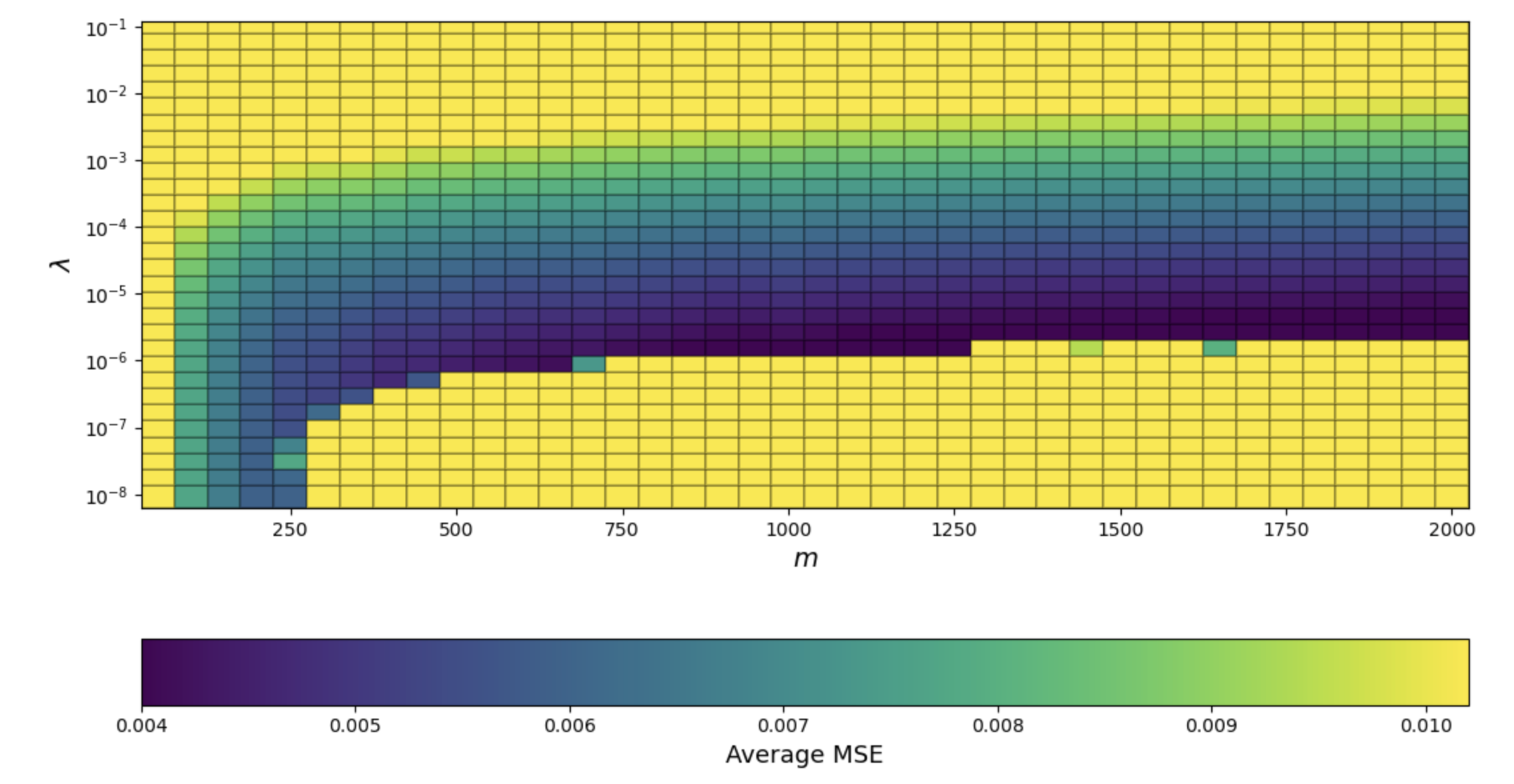}
        \caption{Gaussian kernel}
        \label{HEATMAP_GAUSSIAN}
    \end{subfigure}
    \hfill
    \begin{subfigure}[t]{0.48\linewidth}
        \centering
        \includegraphics[width=\linewidth]{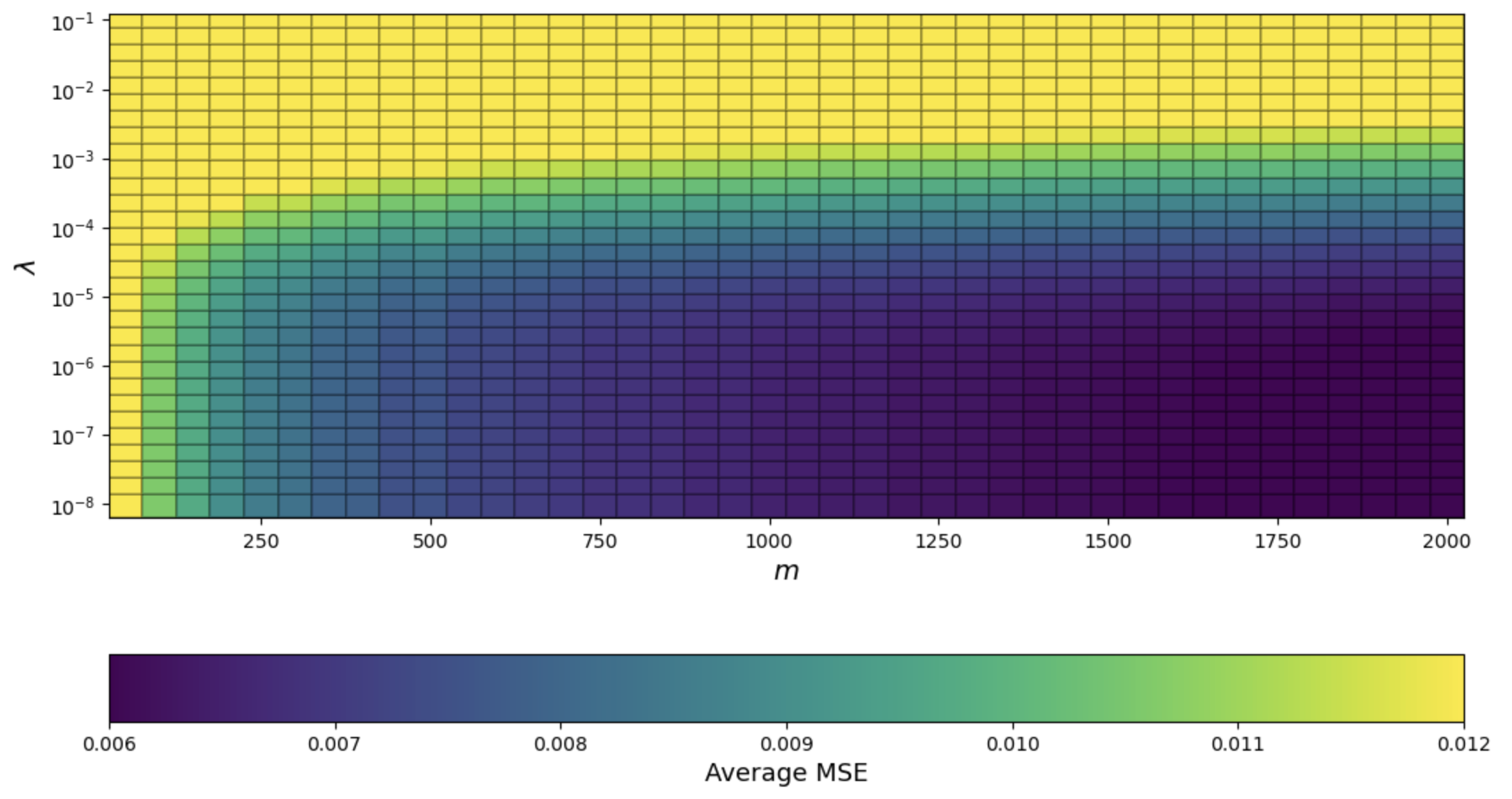}
        \caption{IMQ kernel}
        \label{HEATMAP_IMQ}
    \end{subfigure}
    
    \caption{Heatmap results between $m$ and $\lambda$ for Gaussian and IMQ kernels}
    \label{HEATMAP_COMBINED}
\end{figure}

\newpage

\subsection{Energy Efficiency Dataset}

We now consider the \emph{Energy Efficiency} dataset from the UCI Machine Learning Repository \cite{energy_efficiencyexample_242}. The dataset consists of $768$ simulated building configurations, each described by $8$ input features capturing geometric and architectural properties, including relative compactness, surface area, wall area, roof area, overall height, orientation, glazing area, and glazing area distribution. Using the given data, we apply our proposed operator-learning method to predict the heating and cooling loads of each building. 

Model performance is evaluated using the root mean squared error (RMSE), averaged across both output dimensions. We adopt a data split such that: 80 percent of the data was used for training, while the rest samples are reserved for testing. Hyper-parameters are selected via cross-validation.

We compare the full kernel ridge regression (KRR) solution with its Nyström approximation, which reduces computational complexity by constructing a low-rank approximation of the kernel matrix. In the Nyström method, a subset  is selected uniformly at random from the training data. In addition to RMSE, we record  computational time in order to quantify the efficiency gains achieved by the Nyström approximation relative to the full KRR baseline.  Table~\ref{tab:krr_nystrom_time} and Figure~\ref{fig:krr_nystrom_rmse1} report the prediction accuracy and average computational time to predict each test sample of full KRR and Nyström KRR for varying approximation ranks. The results show that Nyström KRR achieves significant computational savings while maintaining competitive accuracy for moderate values of 
m.

\begin{figure}[ht]
    \centering
    
    \begin{minipage}{0.48\textwidth}
        \centering
        \small
        \begin{tabular}{lccc}
        \toprule
        \textbf{Method} & \textbf{m} & \textbf{RMSE} & \textbf{Time(ms)} \\
        \midrule
        Full KRR        & --   & 0.8992 & 0.000305 \\
        \midrule
        Nyström KRR     & 25  & 3.5513 & 0.000062 \\
        Nyström KRR     & 50   & 3.1140 & 0.000082 \\
        Nyström KRR     & 75   & 2.4779 & 0.000100 \\
        Nyström KRR     & 100  & 1.8203 & 0.000113 \\
        Nyström KRR     & 125  & 1.5254 & 0.000126 \\
        Nyström KRR     & 150  & 1.2117  & 0.000140 \\
        \bottomrule
        \end{tabular}
    \end{minipage}
    \hfill
    \begin{minipage}{0.48\textwidth}
        \centering
        \includegraphics[width=\linewidth]{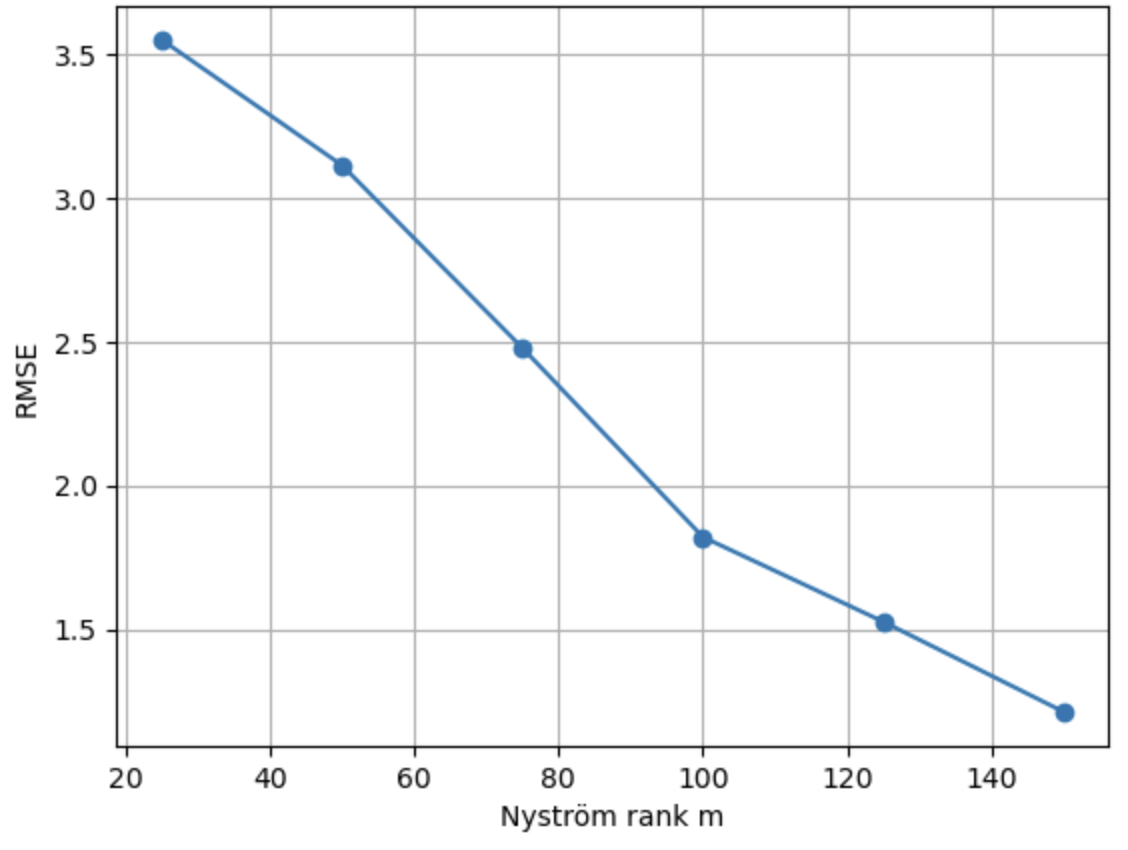}
    \end{minipage}

    \vspace{0.8em}

    \begin{minipage}[t]{0.48\textwidth}
        \captionsetup{type=table}
        \captionof{table}{Comparison of full KRR and Nyström KRR.}
        \label{tab:krr_nystrom_time}
    \end{minipage}
    \hfill
    \begin{minipage}[t]{0.48\textwidth}
        \captionof{figure}{Graph between $m$ and RMSE.}
        \label{fig:krr_nystrom_rmse1}
    \end{minipage}

\end{figure}

\section{ACKNOWLEDGMENT}

Vaibhav Silmana thanks the University Grants Commission (UGC), Government of India, for its financial assistance (Student ID. 221610001628). The authors acknowledge the use of the coding agent Codex for   debugging and verifying the code.

\bibliographystyle{abbrv}\bibliography{ref}
\end{document}